\documentclass{amsart}

\usepackage{amsrefs}

\usepackage{euscript}
\usepackage{amsmath}
\usepackage{amscd}
\usepackage{amssymb}
\usepackage{enumerate}
\usepackage{stmaryrd}
\usepackage{amsfonts}
\usepackage{graphicx}
\usepackage[all]{xy}
\usepackage[margin=1.5in]{geometry}
\usepackage{setspace}
\usepackage{srcltx}

\numberwithin{equation}{section}

\newcommand{\bC}{{\mathbb C}}
\newcommand{\bP}{{\mathbb P}}
\newcommand{\bR}{{\mathbb R}}
\newcommand{\bZ}{{\mathbb Z}}

\newcommand{\sF}{\EuScript F}

\newcommand{\cF}{\mathcal F}

\newcommand{\cH}{\mathcal H}
\newcommand{\cW}{\mathcal W}
\newcommand{\cC}{\mathcal C}

\newcommand{\sL}{\EuScript L}

\newcommand{\Tq}{T_{q}^{*}Q}
\newcommand{\TQ}{T^{*}Q}

\newcommand{\ev}{\operatorname{ev}}

\newcommand{\im}{\operatorname{im}}

\newcommand{\Moduli}{\mathcal M}

\newcommand{\Wrap}{\EuScript{W}}

\newcommand{\Tw}{\mathrm{Tw}}

\newcommand{\Aut}{\operatorname{Aut}}
\newcommand{\Auteq}{\operatorname{Auteq}}
\newcommand{\Symp}{\operatorname{Symp}}
\newcommand{\Laction}{\EuScript{A}}
\newcommand{\Energy}{\EuScript{E}}

\def\co{\colon\thinspace}
\newcommand{\Spin}{\operatorname{Spin}}

\newcommand{\Cone}{{\mathrm {Cone}}}
\renewcommand{\i}{\sqrt{-1}}
\renewcommand{\Tq}[1][\!]{T^{*}_{q_#1} Q_{#1}}
\renewcommand{\L}[1][\!]{L_{q_#1}}
\newcommand{\alg}{\mathrm{alg}}

\newtheorem{thm}{Theorem}[section]
\newtheorem{cor}[thm]{Corollary}
\newtheorem{lem}[thm]{Lemma}
\newtheorem{prop}[thm]{Proposition}
\newtheorem{defin}[thm]{Definition}
\newtheorem{def-lem}[thm]{Definition-Lemma}
\newtheorem{conj}[thm]{Conjecture}

\theoremstyle{remark}

\newtheorem{rem}[thm]{Remark}

\newtheorem{example}[thm]{Example}

\newcommand{\superscript}[1]{\ensuremath{^{\textrm{#1}}} }

\renewcommand{\th}[0]{\superscript{th}}

\newcommand{\noproof}{
\begin{flushright}
\qedsymbol
\end{flushright}}

\newcommand{\comment}[1]{}
\newcommand{\Acknowledgements}{{\em Acknowledgements.} }

\title[Exact Lagrangians in plumbings]{Exact Lagrangians in plumbings}

\author[M.~Abouzaid, I.~Smith]{Mohammed Abouzaid, Ivan Smith} \date{\today}
\thanks{This research was conducted during the period the first author served as a Clay Research Fellow. The second author was partially supported by European Research Council grant ERC-2007-StG-205349. }

\begin{document}

\begin{abstract}
Consider a Stein manifold $M$ obtained by plumbing cotangent bundles of manifolds of dimension greater than or equal to $3$ at points.  We prove that the Fukaya category of closed exact $Spin$ Lagrangians with vanishing Maslov class in $M$ is generated by the compact cores of the plumbing. As applications, we classify exact Lagrangian spheres in $A_2$-Milnor fibres of arbitrary dimension, derive constraints on exact Lagrangian fillings of Legendrian unknots in disk cotangent bundles, and prove that the categorical equivalence given by the spherical twist in a homology sphere is typically not realised by any compactly supported symplectomorphism.
\end{abstract}
\maketitle
\setcounter{tocdepth}{1}

\section{Introduction}

\subsection{Context} \label{sec:context}

A  classical problem in symplectic topology is to understand the topology of exact Lagrangian submanifolds of Stein manifolds. The case which has received most attention is Arnold's ``nearby Lagrangian submanifold" conjecture, asserting that an exact $L \subset T^*M$ should be Hamiltonian isotopic (and in particular diffeomorphic) to $M$.  A  homotopy-theoretic version of this, namely that the inclusion $L\subset T^*M$  is a homotopy equivalence whenever the Maslov class $\mu_L = 0$, has been proved in \cite{Abouzaid-ExactLag}. Moving away from cotangent bundles, much less is known. On the one hand, Dehn twists provide constructions of infinite families of distinct Lagrangian submanifolds in fixed homology or even smooth isotopy classes, for instance in the   $A_n$-surfaces arising as Milnor fibres of the singularities $\bC^2/\bZ_{n+1}$ with $n\geq 2$ \cite{Seidel:graded}.  On the other, there are a few uniqueness or non-existence theorems: 
\begin{itemize}
\item Hind \cite{Hind} proved there is a unique Lagrangian sphere up to isotopy in the $A_1$-space $T^*S^2$; 
\item Ishii, Ueda and Uehara  \cite{IU, IUU} classified Lagrangian spheres in the $A_n$-surfaces up to isomorphism in the Fukaya category, in particular showing that they formed a single orbit under the natural action of the braid group; 
\item Ritter \cite{Ritter} proved that the $A_n$-surfaces contain no exact Lagrangian tori.
\end{itemize}
Hind's work relied on existence results for holomorphic foliations which seem special to four dimensions, Ishii \emph{et al} appealed to sheaf theory computations valid only on complex surfaces, whilst Ritter used non-exact deformations of the symplectic form relying on positivity of the second Betti number.   

On the other hand, the systematic study of the symplectic topology of Stein manifolds has undergone new developments, with progress on at least 3 distinct fronts.  First, the work of Oancea \cite{Oancea:SpecSeq}, Seidel \cite{Seidel:bias}, McLean \cite{McLean} and others on symplectic homology (as defined by Floer-Hofer  and Viterbo) showed the importance of this invariant for distinguishing examples;  second, the work of Bourgeois, Eliashberg and Ekholm \cite{BEE} on Legendrian surgery rendered symplectic homology and its cousins amenable to inductive computation; and finally, there has been progress on foundational and structural results for Fukaya categories \cites{seidel-book, AS, generate}, and their connection with symplectic cohomology.

The purpose of this  note is to illustrate some of these developments, and their effectiveness in constraining the topology of Lagrangian embeddings in Stein manifolds. 

\vspace{0.5cm}

\noindent \textbf{\emph{Convention.}}  We work throughout over a field $k$.  If $k$ is not of characteristic 2, then all Lagrangian submanifolds should be equipped with (relative) Spin structures: note that the zero-section of a cotangent bundle $T^*Q$ is always relatively spin, relative to the background class $w_2(Q)$. As a standing convention throughout the paper, Lagrangian submanifolds are \emph{orientable} and \emph{connected}  (unless explicitly stated otherwise).  A small number of arguments in the paper rely on the freedom to choose the characteristic of $k$ coprime to the index of a covering space; in such situations we restrict to working with $Spin$ Lagrangians, and then the relevant Fukaya categories are defined over $\bZ$.

\subsection{Results} 
Let $Q_0$ and $Q_1$ be smooth Riemannian manifolds.  Recall the \emph{plumbing} $M = D^*Q_0 \# D^*Q_1$ of two cotangent bundles $T^*Q_i$ is given by choosing balls $B_{i}$ in $Q_i$, and identifying their disc cotangent bundles $D^{*} B_{i}  \subset D^*Q_i$  by a symplectomorphism which interchanges the zero-section and fibre directions (see Definition \ref{defin:plumbing} and Section \ref{sec:gener-comp-cores} for more details).  The plumbing $M$ is naturally a Liouville domain, which admits an exact symplectic form for which the obvious inclusions $Q_i \hookrightarrow M$ are exact Lagrangian submanifolds, and can be completed to a Liouville manifold we denote $ T^*Q_0 \# T^*Q_1 $.    There is a trivialisation of the bicanonical bundle $K_M^{\otimes 2}$, well-defined up to homotopy, which restricts on neighbourhoods of $Q_i \subset M$ to the trivialisations arising from complexified volume forms on the Lagrangians $Q_i$.

\begin{thm} \label{thm:main}
Let $M$ be a plumbing of two cotangent bundles. The Fukaya category $\cF(M)$ of closed exact Lagrangians fully faithfully embeds in the subcategory of the wrapped Fukaya category $\cW(M)$ generated by cotangent fibres.  If the real dimension of $M$ is greater than $4$, then every closed exact  Lagrangian with vanishing Maslov class is equivalent to a twisted complex over the components of the compact core equipped with local systems.
\end{thm}

\begin{rem}
The proof generalises to plumbings of many components, at least when the plumbing graph is a tree.  It seems likely that Theorem \ref{thm:main} holds in more generality (for self-plumbings; for plumbings along submanifolds of codimension $\geq 3$), but our argument employs a geometric trick using Lefschetz fibrations (Section \ref{sec:embedd-plumb-lefsch}) which is more restrictive.\end{rem}

 Theorem \ref{thm:main} is a generation, rather than split-generation, result for Fukaya categories, generalising the fact that the zero-section generates the Fukaya category of closed exact Lagrangians in a cotangent bundle.  Twisted complexes, rather than summands thereof, are more readily manipulated algebraically, whence the possibility of new applications.  Whilst these applications are not exhaustive, they should serve to illustrate the reach of Theorem \ref{thm:main}. Statements in the main body of the text are sometimes sharper.
 
 The first application generalises several of the four-dimensional results described in Section \ref{sec:context} to arbitrary dimensions, in the first non-trivial case.
Let $A_2^n$ denote the complex $n$-dimensional $A_2$-Milnor fibre obtained by plumbing two copies of $T^*S^n$, which can also be described as the affine variety
\[
A_2^n \ = \ \left\{ x_1^2 + \cdots + x_n^2 + t^{3} = 1 \right\} \subset \bC^{n+1}
\]
equipped with the restriction of the standard symplectic form $\omega_{\bC^{n+1}} = \frac{i}{2\pi}\sum_{j=1}^{n+1} dz_j \wedge d\bar{z}_j$. 

The zero section in either copy of $T^{*}S^{n}$ defines a Lagrangian sphere in $ A_2^n $, and the two spheres obtained in this way intersect at one point.  The Dehn twists about these spheres generate a subgroup of $\pi_{0}(\Symp(A_{2}^{n})) $ which is isomorphic to the braid group on three strands.  These symplectomorphisms act on the Fukaya category, so we  obtain a subgroup of the group of automorphisms of $ \cF( A_2^n ) $, which we denote $Br^{\alg}_{3}$.  Our next result asserts that, from the point of view of Floer theory, the collection of exact Lagrangians in $A_2^{n}$ is indistinguishable from the collection of spheres obtained by iteratively applying Dehn twists to either component of the skeleton:

\begin{thm} \label{thm:A_2}
Let $n\geq 3$.  Any object of $\cF(A_2^{n})$ corresponding to an exact Maslov zero Lagrangian in $A_{2}^{n}$  lies  in the orbit of the object defined by either component under the group generated by $Br^{\alg}_{3}$ and the shift functor.
\end{thm}
\begin{cor}
  Every exact Maslov zero Lagrangian in $A_{2}^{n}$ is a homology sphere and lies in a primitive homology class.
\end{cor}

\begin{rem} \label{rem:1-connected}
In fact, every exact Lagrangian of vanishing Maslov index in $A_{2}^{n}$ is homotopy equivalent to $S^{n}$. The proof  uses the technology from \cite{Abouzaid-ExactLag}.  We explain the main ideas in Section \ref{sec:exact-lagr-a_2}.
\end{rem}

\begin{rem}
Let $X$ be a Calabi-Yau $n$-fold, either closed or convex at infinity.  Thomas and Yau \cite{ThomasYau} conjectured that a graded Lagrangian submanifold of $X$ with well-defined Floer cohomology flows, under mean curvature flow, to a union of cleanly intersecting special Lagrangian submanifolds.  Whilst the (non-zero!) homology classes of the different components of the special Lagrangian limit can in principle cancel, it seems hard in examples to find nullhomologous configurations which could ever smooth to \emph{spheres} (at least in dimension 3, where smoothing tends to push up the first Betti number).  Since spheres in Calabi-Yau manifolds are always unobstructed \cite{FO3}, one is led to ask if every Lagrangian sphere in a Calabi-Yau is homologically essential.  The Corollary proves this in the first non-trivial case beyond the affine quadric $T^*S^n$.  \end{rem}

In our next application, we consider the \emph{Legendrian unknot}  $\Lambda \subset S^*Q = \partial D^*Q $ which is the boundary of a disk cotangent fibre in $D^*Q$.   We say the Weinstein surgery $M_Q$ obtained by adding an $n$-handle to $D^*Q$ along  $\Lambda$ is obtained by \emph{capping} the unknot.  The Legendrian $\Lambda$ naturally bounds a Lagrangian disk in $M_{Q}$ lying on the outside of $D^*Q$; in particular, any Lagrangian filling $K \subset D^*Q$ of $\Lambda$, with $\partial K = \Lambda$, defines a closed Lagrangian submanifold $\bar{K} \subset M_Q$.  The simplest case arises from choosing $K$ to be the cotangent fibre, yielding a Lagrangian sphere $S_{\Lambda}$ in  $M_Q$.

Let $\Sigma$ be an integral homology sphere with non-trivial fundamental group.  Let $M_{\Sigma}$ be obtained by capping the unknot in $D^*\Sigma$.  

\begin{thm} \label{thm:hlgy_sphere}
If $\bar{K} \subset M_{\Sigma}$ is obtained by capping an exact Maslov zero filling $K$ of $\Lambda$, then $\bar{K}$ is equivalent in $\cF(M_{\Sigma})$ either to the sphere $S_{\Lambda}$, or to a Lagrange surgery of $S_{\Lambda}$ with the zero-section $\Sigma$.
\end{thm}

\begin{cor}
Every exact Maslov zero filling of the Legendrian unknot in $ T^{*} \Sigma $   is a homology ball.
\end{cor}

Graded symplectomorphisms act naturally on the Fukaya category.  Since $T^*\Sigma$ admits a unique trivialisation of its bicanonical class up to homotopy, arbitrary compactly supported symplectomorphisms of $T^*\Sigma$ induce autoequivalences of $\cF(M_{\Sigma})$, well-defined up to the group $\langle [1] \rangle  $ generated by the shift functor.

\begin{cor}\label{cor:not-onto}
The image of the natural map $\rho: \pi_0(\Symp_{ct}(T^{*} \Sigma)) \to \Aut(\cF(M_{\Sigma})/\langle [1] \rangle$ intersects the $\bZ$-subgroup defined by the spherical twist in $\Sigma$ only in the identity. 
\end{cor}

Our results imply the image of $\rho$ is torsion-free, and presumably the (wrapped) Fukaya category of the cotangent bundle cannot be used to detect torsion classes in $\pi_0 \, \Symp_{ct} (T^*\Sigma)$ (should such exist).

\begin{rem} \label{rem:stabilise}
In \cite{kontsevich-09}, Kontsevich introduced a notion of stabilisation (by dimension) of a symplectic manifold, and asked whether the map from the group of components of the symplectomorphism group to the autoequivalences of the Fukaya category becomes an isomorphism after sufficiently many stabilisations.  In Section \ref{sec:spher-twists-homol} we shall explain that  (the argument underlying) Corollary \ref{cor:not-onto} implies that the answer is negative, assuming that the Fukaya categories of stabilisations -- which have not been defined or studied explicitly in the literature -- behave in the expected manner under passing to covers, cf. Conjecture \ref{conj:covers}.
\end{rem}

Spherical twists are one of the few known sources of interesting autoequivalences of Fukaya categories. The underlying geometric Dehn twist in a Lagrangian sphere exists because the sphere admits a metric with periodic geodesic flow.  Similar metrics exist on other compact rank one symmetric spaces (complex projective spaces, the Cayley plane, etc), and for such a manifold $Q$, the category $\cF(M_Q)$ again has infinite-order autoequivalences that are non-trivial modulo shift.  A classical result of Bott \cite{Bott} states that if $Q$ admits a metric with periodic geodesic flow, then it has, as cohomology algebra, a truncated polynomial ring $H^*(Q) \cong \bZ[u] /  \langle u^k \rangle$.  It seems likely that there is a Fukaya-categorical version of Bott's theorem.

\begin{conj}
Let $Q$ be simply connected. If the map $\rho$ of Corollary \ref{cor:not-onto} is non-zero then $H^{*}(Q)$ is a truncated polynomial ring.  For any $Q$, the image of $\rho$ is trivial or $\bZ$.
\end{conj}

There is a more natural formulation of the Conjecture, bypassing the auxiliary manifold $M_Q$ but invoking a version of the Fukaya category of the cotangent bundle which incoporates certain non-closed Lagrangian submanifolds Legendrian at infinity, for instance the Nadler-Zaslow category \cite{NZ}. The version above is tailored to our technical machinery. 
We prove two partial results.  We will say that the representation $\rho$ is \emph{categorically trivial} if any autoequivalence in the image of $\rho$ acts trivially on objects of $\cF(M_Q)$.

\begin{prop} \label{prop:partial_result_truncated}
If $Q$ is a product of spheres, or if $Q$ has homology concentrated in $3$ degrees and the middle Betti number is greater than $1$, then $\rho$ is categorically trivial.  
\end{prop}

In particular, if $Q$ is a simply-connected 4-manifold and $M_Q$ admits a non-trivial autoequivalence, then $Q$ is homotopy equivalent to one of $S^4$ and $\bC\bP^2$ (this is essentially sharp, in that each of these manifolds does admit a non-trivial Dehn twist).  Wendl \cite{Wendl} used delicate constructions of planar holomorphic foliations to prove that $\Symp_{ct}(T^*(S^1\times S^1))$ is connected, and the first case of Proposition \ref{prop:partial_result_truncated} can be viewed as a partial generalisation of that result.

\vspace{0.5cm}
\noindent \Acknowledgements  The authors are grateful to Tobias Ekholm, Gabriel Paternain and Paul Seidel for helpful conversations, and to the anonymous referee, whose detailed comments located and corrected several small errors.

\section{Proof of the applications}
\label{sec:proof_applications}

In this section, we derive the consequences of Theorem \ref{thm:main} which are mentioned in the introduction.  We shall consider, throughout, a minimal model $H^{*} \cF(M)$ for the $\bZ$-graded ordinary Fukaya category of a Liouville domain $M$.  This is a strictly unital $A_{\infty}$ category, whose objects are \emph{ closed Lagrangian branes}, i.e. closed connected exact Lagrangians, equipped with a $\Spin$ structure and a grading in the sense of \cite{Seidel:graded}, and whose morphism spaces are Floer cohomology groups.  We shall use the notation
\begin{equation}
  \Tw(H^{*}(\cF(M)))
\end{equation}
for the category of twisted complexes, an object of which consists of a sequence of Lagrangian branes $\{ L_{i} \}_{i=1}^{D}$, carrying finite-dimensional local systems $\{ V_{i} \}_{i=1}^{D}$ -- equivalently modules for the group ring $\bZ[\pi_1(L_i)]$ -- concentrated in a single degree, and degree $1$ maps
\begin{equation}
  \delta_{i,j} \in  HF^{*}(V_i,V_j)  
\end{equation}
satisfying $\sum_{r\geq 1} \mu^r (\delta, \ldots, \delta) = 0$, with $\mu^r$ the $A_{\infty}$-operations of the additive enlargement of $H\cF(M)$ and $\delta = (\delta_{ij})$ the upper triangular matrix of differentials. Whenever the local systems are trivial, we may alternatively write $\delta_{i,j}$ as a map
\begin{equation}
  \label{eq:differential}
  \delta_{i,j} \co V_{i} \to HF^{*}(L_i,L_j)  \otimes V_{j}.
\end{equation} This category admits a shift functor
\begin{equation}
  X \mapsto X[1]
\end{equation}
which decreases the  grading of each local system $V_i$ by $1$. For background on the category of twisted complexes, see \cite[Chapter 1, Section 3]{seidel-book}.   Whenever $\Sigma \subset M$ is an exact Lagrangian submanifold of a Liouville manifold $M$,  there is an associated algebraic twist functor
\[
T_{\Sigma} \in \textrm{NuFun}(H\cF^*(M))
\]
defined as usual by the cone of evaluation 
\begin{equation} \label{eqn:twist}
T_{\Sigma} = \ Cone \left ( HF(\Sigma,\bullet) \otimes \Sigma \stackrel{ev}{\longrightarrow} \bullet \right ).
\end{equation}
If $H^*(\Sigma) \cong H^*(S^n)$, the twist functor is invertible, hence defines an autoequivalence of $H^*(\cF(M))$.  Explicitly, the inverse functor is given by the cone on the dual of evaluation
\begin{equation} \label{eqn:inverse-twist}
T_{\Sigma}^{-1} = \ Cone \left ( \bullet \stackrel{ev^{\vee}}{\longrightarrow} HF(\bullet, \Sigma)^{\vee} \otimes \Sigma \right),
\end{equation}
compare to \cite[Remark 5.12]{seidel-book}.

In the body of the paper, most of our Liouville manifolds $M$ will arise as \emph{plumbings}. Define the \emph{model plumbing} to be the open neighbourhood  
\begin{equation} \label{eqn:model-region}
\{|x| |y| \leq 1\} = R^{\infty} \subset \bC^n \ = \ \bR^n \times \i \bR^n
\end{equation}
 of the union of transverse linear Lagrangian planes $\bR^n \cup \i \bR^n \, = \, \{y=0\} \cup \{x=0\}$.  Now given base-points $q_i \in Q_i$,  choose  Riemannian metrics $g_i$ locally flat near $q_i$, and pick isometric embeddings $\eta_i:  (B(1/2), 0) \hookrightarrow (Q_i, q_i)$ of the ball of radius 1/2 in $\bR^n$ to $Q_i$. There are then natural symplectic inclusions 
\[
D\eta_i: \, D^*B(1/2) \hookrightarrow D^*Q_i.
\]
On the other hand, inclusions $\iota_0: B(1/2) \subset \bR^n$ and $\iota_1: B(1/2) \subset \i  \bR^n$ induce symplectic inclusions $D\iota_j : D^*B(1/2) \hookrightarrow R^1 \subset \bC^n$.

\begin{defin} \label{defin:plumbing}
The \emph{plumbing} $M=T^*Q_0 \# T^*Q_1$ is the Liouville manifold obtained by completing a Weinstein domain defined by identifying points in the image of $D\eta_0$ with points in the image of $D\eta_1$ whenever their $D\iota_j$-images in $R^{\infty} \subset \bC^n$ co-incide.  
\end{defin}

Essentially, we are identifying the two copies of $D^*B(1/2) \cong B^n \times B^n$ via the map $(x,y) \mapsto (y,-x)$, which preserves the symplectic form $dx\wedge dy$ (the result of the identification inherits a natural Liouville structure).  

\subsection{Exact Lagrangians in $A_{2}$ Milnor spaces}
\label{sec:exact-lagr-a_2}

\noindent \emph{Throughout this section, fix a dimension $n\geq 3$.  We work over an arbitrary coefficient field $k$.}  

Let $A_2^n$ denote the $n$-dimensional $A_2$-Milnor fibre obtained by plumbing two copies of $T^*S^n$, which can also be described as the affine variety
\[
A_2^n \ = \ \left\{ x_1^2 + \cdots + x_n^2 + t^{3} = 1 \right\} \subset \bC^{n+1}
\]
equipped with the restriction of the standard symplectic form $\omega_{\bC^{n+1}} = \frac{i}{2\pi}\sum_{j=1}^{n+1} dz_j \wedge d\bar{z}_j$. This section will classify exact Lagrangian submanifolds of Maslov class zero inside $A_2^n$ up to quasi-equivalence in the Fukaya category.   The main technical step applies in slightly more generality, so we begin with a Stein manifold $(M,\omega)$ which is the plumbing of two cotangent bundles $T^*Q_i$ for $\bZ$-homology $n$-spheres $Q_0$ and $Q_1$.  We fix gradings up to a global shift by requiring that
\begin{equation} \label{eqn:fix-grading}
HF(Q_0, Q_1) \cong k \quad \textrm{is concentrated in degree} \ 1.
\end{equation}
and consider $Q_i$ as objects of $H^{*}\cF(M)$.  The category $H^{*} \cF(M)$  admits an action of the braid group $Br_3^{\alg}$ on three strings, obtained from the twist functors $T_i$ in the spherical objects $Q_i$.  An easy consequence of the definition of the algebraic Dehn twist is the following result from \cite[Lemma 5.11] {seidel-book}.
\begin{lem} \label{lem:shift}
$T_i(Q_i)$ is quasi-equivalent, in $  \Tw(H^{*}\cF(M))  $, to  $Q_i[1-n]$. \noproof
\end{lem}

\begin{prop} \label{prop:braid-orbit}
Suppose $\pi_1(Q_i)$ admits no non-trivial finite-dimensional representation, for $i=0,1$. Up to quasi-equivalence and shift in $ \Tw(H^{*} \cF(M))$, every closed Lagrangian brane lies in the braid group orbit of  $Q_0$ or $Q_1$. 
\end{prop}

\begin{rem} \label{rem:cores-in-same-orbit}
In fact $Q_1$ and $Q_{0}$ lie in the same orbit of the braid group action.
\end{rem}

We will break up the proof  into stages.  Given a Lagrangian brane $L$, generation of $\cF(M)$ by the cores (Theorem \ref{thm:main}) implies $L$, as an object of $  \Tw(H^{*}(\cF(M)))  $ is equivalent to a twisted complex built from local systems on the Lagrangians $Q_i$.
\begin{lem}
  There is a twisted complex which is quasi-isomorphic to $L$, built from the objects $Q_i$, and in which none of the arrows are identity morphisms.
\end{lem}
\begin{proof}
The hypothesis on the fundamental groups implies that the modules $V_i$ appearing in this complex are trivial, i.e. just finite-dimensional $k$-vector spaces.   Moreover, the compact cores $Q_i$ define rank one simple modules over the wrapped Floer cohomology algebra of the cotangent fibres $T_{q_i}^*Q_i$.  Proposition \ref{prop:non-positive} guarantees that this algebra is non-positively graded,  and Lemma \ref{lem:degree-at-least-2} asserts that in a twisted complex of (ordinary) modules over an $A_{\infty}$-algebra graded in non-positive degree, all arrows arise from morphisms of degree at least 2 in the $A_{\infty}$-algebra. 
\end{proof}

 We shall use throughout this section the fact that, in the subcategory of $ H^{*}(\cF(M)) $ with objects $Q_i$, all higher products vanish for degree considerations (recall that our minimal model is assumed strictly unital;  the vanishing result is in fact true whenever $n \geq 2$, but $n>2$ is the easy case).   Proposition \ref{prop:braid-orbit} is a consequence of the following purely algebraic result:
\begin{lem} \label{lem:classify_positive_endo_algebra}
Under the same hypotheses as in Proposition \ref{prop:braid-orbit}, let $\cC^{\bullet}$ be a  twisted complex over $Q_{0}$ and $Q_1$ whose endomorphism algebra is supported in non-negative degrees. Then $\cC^{\bullet}$ is in the braid group orbit of a direct sum of copies of $Q_j$ supported in a single degree, for either $j=0$ or $j=1$.
\end{lem}

By applying a suitable shift, we assume that the twisted complex $\cC^\bullet$ is concentrated in degrees $0 \leq i \leq N$ for some $N\geq 0$. The degree $i$ part of the complex is 
\[
U_i \otimes Q_0 \ \oplus \ V_i \otimes Q_1
\]
where $U_i, V_i$ are $k$-vector spaces concentrated in degree $i$; the internal differential (of degree $+1$) in the twisted complex is comprised of arrows $U_i \rightarrow V_i$, $V_j \rightarrow U_{j-n+2}$,  $U_i \rightarrow U_{i-n+1}$ and $V_i \rightarrow V_{i-n+1}$, labelled by the appropriate generators of $HF(Q_i, Q_j)$.  Diagrammatically (taking $n=4$):

\begin{center}
\begin{equation} \label{Eqn:LasTwistedComplex}
\xymatrix{
U_0 \ar[d] & U_1  \ar[d] & U_2 \ar[d]  & U_3  \ar[d] \ar@/_2pc/[lll] &   \ar@/_2pc/[lll] \cdots \\
V_0 & V_1 & V_2  \ar@{-->}[ull]  & V_3 \ar@{-->}[ull]  \ar@/^2pc/[lll] &  \ar@/^2pc/[lll] \cdots
}
\end{equation}
\end{center} \vspace{1cm}

\noindent with vertical arrows labelled by $HF(Q_0, Q_1)$, dotted arrows by $HF(Q_1, Q_0)$ and the curved arrows by the fundamental classes of the $Q_i$ (note that Equation \ref{eqn:fix-grading} implies that the total differential is indeed degree $1$).  

\begin{rem} \label{rem:dimension} 
The curved arrows emanate only from $U_j$ and $V_j$ with $j\geq n-1 \geq 2$.
\end{rem}

The key ingredient is a measure of complexity of $\cC^{\bullet}$ defined as follows.  

\begin{defin}
Let $|U_i| = 2i+1$, $|V_j| = 2j$  and define 
\[
cx(\cC^{\bullet} ) = \max \left( \max_{U_i \neq 0} |U_i|, \, \max_{V_i \neq 0} |V_i| \right) - \min \left( \min_{U_i \neq 0} |U_i|,\,  \min_{V_i \neq 0} |V_i| \right).
\]
\end{defin}

In other words, the complexity is the longest zig-zag path $\downarrow \nwarrow \downarrow \nwarrow \cdots$ that can be drawn in $\cC^{\bullet}$. Note that we are \emph{not} assuming that all the arrows in the zig-zag are represented by non-trivial differentials in the complex (the $\nwarrow$ here is a map $V_i \rightarrow U_{i-1}$ whose degree would not be compatible with a differential), just that the complex stretches out over the requisite number of degrees.  We will prove by induction on $cx(\cC^{\bullet})$ that $\cC^{\bullet}  \simeq \sigma(Q_0)$ for some $\sigma \in Br_3^{\alg} = \langle T_0, T_1\rangle$.

\begin{lem} \label{lem:first-step}
Assume $U_0 \oplus V_0 \neq 0$.  Let $N = max_i \left\{U_i \oplus V_i \neq 0 \right\}$. Then 
\[
V_0 \neq 0 \, \Leftrightarrow \, U_N \neq 0.
\]
\end{lem}

\begin{proof}
Suppose for contradiction that $V_0 \neq 0$, but $U_N = 0$. Consider the component of $End_{H\cF(M)}(\cC^{\bullet}, \cC^{\bullet})$ containing $Hom(V_N, V_0) \otimes e_{Q_1}$.   Since the twisted complex has no arrows out of $V_0$, and no arrows into $V_N$ by the hypothesis $U_N = 0$, this morphism must be closed.  The fact that $\mu^k = 0$ for $k\geq 3$ and that there are no identity morphisms $e_{Q_1}$ in the differentials of the twisted complex furthermore implies that $Hom(V_N,V_0)\otimes e_{Q_1}$ cannot be exact: if $\mu^2(p,q) = e_{Q_1}$ for differentials $p,q$ then $p=e_{Q_1} = q$.  However, this morphism has negative degree, which is a contradiction to our assumption that
\[
End_{\Tw(H\cF(M))}(\cC^{\bullet}, \cC^{\bullet}) 
\]
is concentrated in non-negative degrees.  Therefore, $V_0 \neq 0 \Rightarrow U_N \neq 0.$ The reverse implication is analogous, looking instead at $Hom(U_N, U_0) \otimes e_{Q_0}$ and recalling that $e_{Q_0}$ does not appear as a differential.
\end{proof}

\begin{rem} \label{rem:first-step-enhanced}
The proof of Lemma \ref{lem:first-step} actually yields a slightly stronger conclusion:
\begin{enumerate}
\item If $U_N = 0$ then $V_j = 0$ for $j<n-2$;
\item If $V_0 = 0$ then $U_{N-j}=0$ for $j<n-2$.
\end{enumerate}
This is because, say in the first case, the proof $U_N = 0 \Rightarrow V_0 = 0$ only used the fact that no arrows come out of $V_0$, which also holds true for $V_j$ with $j<n-2$.
\end{rem}

\begin{lem} \label{lem:second-step-1}
In the situation of Lemma \ref{lem:first-step}, suppose $V_0 \neq 0$.  Then either  $V_{N-i} = 0$ for $i<n-1$ or  $U_j = 0$ for $ j < n-1$.
In the first case $T_0^{-1}( \cC^{\bullet} )$ has lower complexity than $\cC^{\bullet}$ and in the second $T_1( \cC^{\bullet})$ has lower complexity than $\cC^{\bullet}$.
\end{lem}

\begin{proof}
Suppose first that $N \geq n-1$.  We will prove that one can reduce the complexity by applying one of the moves indicated.  Iteratively, one can then simplify the twisted complex by a braid until $N \leq n-2$. At this point, there are no curved arrows; in this special case, a minor modification of the arguments given below allows one to further reduce complexity directly.   

The start of the analysis applies to both cases.  We claim that
\begin{itemize}
\item $U_i \hookrightarrow V_i$ is injective for $i=0,1,\ldots, n-2$;
\item $V_{i+n-2} \rightarrow U_i$ vanishes for $i=0,1, \ldots, n-2$.
\end{itemize}
For the first statement, suppose the vector space morphism $U_i \rightarrow V_i$ has kernel, with $i<n-1$.  Pick a non-trivial element $\gamma \in Hom(U_N, U_i)$ with image in that kernel, noting that $U_N \neq 0$ by Lemma \ref{lem:first-step}.  Then $\gamma \otimes e_{Q_0}$ is closed (since no other arrows come out of $U_i$; this follows from Remark \ref{rem:dimension}) but not exact, being a multiple of an identity element. It survives to the cohomology of endomorphisms of $\cC^{\bullet}$, but has negative degree, a contradiction.  Vanishing of $V_{i+n-2} \rightarrow U_i$ when $i<n-1$ now follows from the fact that  $d_{\cC^{\bullet}} \circ d_{\cC^{\bullet}} = 0$ in the twisted complex $\cC^{\bullet}$. Indeed, the $Hom(V_{i-n+2}, V_i)$-term in the Maurer-Cartan equation for $d_{\cC^{\bullet}}$ is a tensor product $Hom(V_{i-n+2}, U_i) \otimes Hom(U_i, V_i)$.  Analogously, we claim that
\begin{itemize}
\item $U_{N-i} \twoheadrightarrow V_{N-i}$ is surjective for $i=0,1, \ldots, n-2$;
\item $V_{N-i} \rightarrow U_{N-i-n+2}$ is zero for $i=0,1, \ldots, n-2$.
\end{itemize}
The first assertion follows since if the map had cokernel, one would split $V_{N-i} \cong im(U_{N-i}) \oplus V_{N-i}'$ and consider an element of $Hom(V_{N-i}',V_0)\otimes e_{Q_1}$ to obtain a cohomologically essential endomorphism of negative degree.  Closedness again follows from Remark \ref{rem:dimension}, and the fact that $i < n-1$. The second assertion then follows from $d^2=0$ for the twisted complex $\cC^{\bullet}$, as previously; the $Hom(U_{N-i}, U_{N-i-n+2})$-component in the Maurer-Cartan equation comes from morphisms factoring through $V_{N-i}$.

We now apply the autoequivalence $T_0^{-1}$ to $\cC^{\bullet}$, cf. \eqref{eqn:inverse-twist}.  Since $U_{N-i} \rightarrow V_{N-i}$ is onto, we can split the map to write $U_{N-i} \cong U_{N-i}' \oplus V_{N-i}$. By Lemma \ref{lem:shift}, applying $T_0^{-1}$  has the effect of shifting the $U_{N-i}'$-term of the complex leftwards by $n-1$ places, and replacing the $V_{N-i} \otimes (Q_0 \oplus Q_1)$ term by $V_{N-i} \otimes Q_1$ only. In particular, after applying $T_0^{-1}$, the top right $n-2$ places in the complex $T_0^{-1}(\cC^{\bullet})$ are zero. One can apply a similar analysis at the left hand side of the original complex $\cC^{\bullet}$.  Using the fact that $U_i \subset V_i$ if $i<n-1$, and writing $V_i = V_i' \oplus im(U_i)$, the autoequivalence $T_0^{-1}$ has the effect of replacing $U_i \otimes (Q_0 \oplus Q_1)$ by $U_i \otimes Q_1$ and $V_i' \otimes Q_1$ by $V_i' \otimes Q_1 \rightarrow V_i' \otimes Q_0[2-n]$.  The upshot is that one has a twisted complex of the shape (for the diagram, $n=4$ again, and $\dagger, \ast$ denote vector spaces which can vary from instance to instance, or depend on the value of $N$): \newline

\begin{center}
\begin{equation} \label{eqn:model-complex}
\xymatrix{
V_0'  & V_1' & \dagger \oplus V_2'  \ar[d] & \ast \ar[d]   \ar@/_2pc/[lll] &  \ast \ar[d] \ar@/_2pc/[lll] & \cdots & 0 & 0 \\
0 & 0& V_0 \ar@{-->}[ull] & V_1 \ar@{-->}[ull] & V_2  \ar@{-->}[ull]  &  \ar@/^2pc/[lll] \cdots & V_{N-1} & V_{N}  \ar@{-->}[ull] 
}
\end{equation}
\end{center} \vspace{1cm}
We now consider the two cases of the statement of the Lemma. 

\begin{enumerate}

\item  If $V_{N-i}=0$ for every $i=0,1\ldots, n-2$, then although the complex now extends an additional $n-2$ columns to the left, the overall complexity has been reduced from $2N+1$ to $\leq 2N$, hence by at least $1$ (essentially because of the gap at the bottom left of the diagram above).  

\item If some $V_{N-j} \neq 0$ with $0 \leq j \leq n-2$, then we claim that the diagonal arrows $V_k \dashrightarrow V_k'$ are injective for $0 \leq k \leq n-2$.  Otherwise, in the usual fashion, we can build a negative degree cohomologically essential morphism by taking an element of $Hom(V_{N-j}, V_k) \otimes e_{Q_1}$.  An additional step is needed in case $k=n-2$, since the vector space occuring at that point in the top row of \eqref{eqn:model-complex} is $\dagger \oplus V_{n-2}'$ which may be larger than $V_{n-2}'$.  Namely, we must observe that there is no non-trivial component $V_{n-2} \rightarrow \dagger$:  indeed, $\dagger = \ker(U_{n-1} \rightarrow V_{n-1})$.  Since $V_k'  $ is the quotient of $V_k$ by $U_k$, we therefore conclude that $U_k = 0$ for $0 \leq k \leq n-2$.  In other words, our initial twisted complex is of the shape \newline
 
 \begin{center}
\begin{equation}
\xymatrix{
0  & 0  & 0   & U_3 \ar[d] & U_4 \ar[d]     &    \cdots & U_{N-1} \ar[d] & U_N \ar[d] \\
V_0 & V_1 & V_2  & V_3 \ar@/^2pc/[lll] & V_4    \ar@/^2pc/[lll] &  \ar@/^2pc/[lll] \cdots & V_{N-1} & V_{N}  \ar@{-->}[ull] 
}
\end{equation}
\end{center} \vspace{1cm}

We now claim that the vertical arrows $U_{N-j} \rightarrow V_{N-j}$  must be isomorphisms for $0 \leq j \leq n-2$ at the top right of the diagram (we already know they are surjective).  If this was not true, then take an element in the kernel of one of these maps, and consider $Hom(U_{N-j}, V_0) \otimes \langle p
\rangle$ to obtain a closed and non-exact endomorphism of negative degree, where $p\in HF(Q_0, Q_1)$ is the unique intersection point.  Finally, apply the autoequivalence $T_1$ to the resulting complex.  The terms $V_j$ for $j<n-1$ are shifted rightwards, so the leftmost $n-2$ columns now vanish; whilst the isomorphisms $U_{N-j} \rightarrow V_{N-j}$ in the rightmost columns are replaced by terms $U_{N-j}$ on the upper row, and $V_{N-j-n+1}$ on the lower row.  Thus, we see \newline

 \begin{center}
\begin{equation}
\xymatrix{
0  & 0  & 0   & U_3 \ar[d] & U_4 \ar[d]   &     \cdots & U_{N-1} \ar[d] & U_N \ar[d] \\
0 & 0 & 0   & V_0' & V_1'   &   \cdots   & V_{N-3}  & V_{N-2}  \ar@{-->}[ull] 
}
\end{equation}
\end{center} \vspace{1cm}

\noindent The complexity has again been reduced by at least one, which completes the proof of the second case.
\end{enumerate}
\end{proof}

\begin{lem} \label{lem:second-step-2}
In the situation of Lemma \ref{lem:first-step}, suppose $V_0 =0$.  Then either $V_{i} = 0$ for $i=0,1,\ldots, n-2$ or $U_{N-i} = 0$ for $i=0,1,\ldots, n-2$;  the twist $T_1^{-1}(\cC^{\bullet})$ has lower complexity than $\cC^{\bullet}$ in the first case, and $T_0(\cC^{\bullet})$ has lower complexity in the second case.  \end{lem}
\begin{proof}
The proof is exactly analogous to that of Lemma \ref{lem:second-step-1}, and indeed reduces to that Lemma by suitable relabelling.  More precisely, we replace the ordered pair $(Q_0, Q_1)$ by the pair $(Q_1[n-2], Q_0)$, noting that in both cases the unique morphism from the first brane to the second is concentrated in degree 1.  The effect on the complex \eqref{Eqn:LasTwistedComplex} is to shift the lower row to the left, so the dotted arrows become vertical, and the previously vertical arrows move $n-2$ places to the left.  Finally, since we have assumed throughout that $U_N \oplus V_N \neq 0$, and the hypothesis $V_0 = 0$ implies that $U_N = 0$ by Lemma \ref{lem:first-step}, we know that $V_N \neq 0$. Up to reversing the roles of the $U_i$ and $V_i$, this brings us back into exactly the situation of Lemma \ref{lem:second-step-1}.  Moreover, Remark \ref{rem:first-step-enhanced} implies that the relabelling operation does not itself increase complexity.    The conclusion of Lemma \ref{lem:second-step-1}, given the relabelling, is therefore to say that one can reduce complexity either by applying the inverse twist in $Q_1[n-2]$ or the twist in $Q_0$.  The twist functor in a shifted object is exactly the same as the twist functor in the unshifted object, which completes the proof.
\end{proof}

Applying  Lemmata \ref{lem:second-step-1}, \ref{lem:second-step-2}, we immediately conclude Lemma \ref{lem:classify_positive_endo_algebra}.
\begin{proof}[Proof of Proposition \ref{prop:braid-orbit}]
Since every exact Lagrangian has endomorphism algebra supported in degree $[0,n]$, Lemma \ref{lem:classify_positive_endo_algebra} implies that such a Lagrangian is in the braid group orbit of a direct sum of copies of one of the components of the skeleton.  (A choice of component is immaterial, by Remark \ref{rem:cores-in-same-orbit}.)  Whenever this direct sum has more than one factor, the endomorphism algebra has dimension $>1$ in degree $0$.  We conclude that a connected exact Lagrangian must in fact be in the braid group orbit of (either of) the two components of the skeleton.
\end{proof}
As an immediate Corollary, one has:
\begin{cor} \label{cor:homology}
Let $M$ be the plumbing of two $\bZ$-homology spheres $Q_i$ for which the groups $\pi_1(Q_i)$ admit no non-trivial finite-dimensional representation.  A Maslov zero exact Lagrangian submanifold of $M$ is a $\bZ$-homology sphere, which moreover realises a primitive homology class.
\end{cor}

\begin{rem}
If $Q_0 \cong S^n \cong Q_1$, then the autoequivalences $T_i$ are realised by the Dehn twists in the $Q_i$.  Later we will prove that if $\pi_1(Q_i) \neq \{1\}$, then $T_i$ is not induced by any compactly supported symplectomorphism, and the  images of the $Q_i$ by general braids are \emph{not} represented by embedded Lagrangians.
\end{rem}

\begin{rem} \label{rem:k-homology-sphere}
If the $Q_i$ are $k$-homology spheres for some field $k$, with the same hypotheses on $\pi_1(Q_i)$, then the conclusion of Corollary \ref{cor:homology} holds in the Fukaya category linear over $k$; in particular, every exact $L\subset M$ of Maslov class zero is then a $k$-homology sphere.
\end{rem}

One can vary the hypotheses on the $\pi_1(Q_i)$ somewhat.  In general, the modules $V_i$ appearing in the twisted complex $\cC^{\bullet}$ may be non-trivial local systems. Morphisms between such are still supported in non-negative degrees.  The proof relies on triviality of the $V_i$ only for this statement, and for the fact that the horizontal arrows in the complex $\cC^{\bullet}$ are all obtained from the fundamental classes of the $Q_i$.  For non-trivial $V_i$, there may in general be additional classes in $HF(V_i \otimes Q_i, V_j \otimes Q_j)$ in intermediate degree.  However, if such classes vanish under pullback to a covering space, one can use the methods of Section \ref{sec:spher-twists-homol}, and in particular Theorem \ref{thm:category-of-cover}, to eliminate the contributions of these classes.  As a special case:

\begin{cor}
If each $Q_i$ either admits no non-trivial finite-dimensional local systems, or has contractible universal cover, then the conclusions of Corollary \ref{cor:homology} hold. \end{cor}

We will leave the interested reader to fill in the details.

\begin{rem}
  Note that the proof of Lemma \ref{lem:classify_positive_endo_algebra} did not use at any stage the fact that the vector spaces $V_{i}$ and $U_{i}$ are finite dimensional.  In particular, if one knows that every local system (of arbitrary rank) on a Lagrangian lies in the category generated by (possibly infinite) direct sums of the components of the skeleton, the same argument would show that any such object in the Fukaya category of a plumbing would be equivalent to an object in the braid group orbit of direct sums of $Q_0$.  Quasi-isomorphism in the wrapped category implies quasi-isomorphism in the usual category of $dg$-local systems, by results in \cite[Appendix B]{Abouzaid-ExactLag}. In particular, we would conclude that every representation of $\pi_1(L)$ would split as a direct sum of trivial ones.  As we note in Remark \ref{rem:generation_infinite_local_systems}, the results of \cite{Abouzaid-ExactLag } imply such a generation statement, from which we conclude the simple connectivity statement asserted in the introduction, Remark \ref{rem:1-connected}.
\end{rem}

\subsection{Spherical twists in homology spheres}
\label{sec:spher-twists-homol}

The invertibility of the spherical twist functor $T_{\Sigma}$ relies only on the fact that $H^*(\Sigma) \cong H^*(S^n)$. If $\Sigma \cong S^n \subset M$ is diffeomorphic to the sphere, then there is also a geometric Dehn twist $\tau_{\Sigma}$. This admits a canonical lift to a graded symplectomorphism, hence acts on $\cF(M)$, and Seidel (following Kontsevich) proved that \cite[Corollary 17.17]{seidel-book}
\[
T_\Sigma \ \cong \  \tau_{\Sigma}
\]
are quasi-isomorphic in the category of $A_{\infty}$-endofunctors of $\cF(M)$.  By contrast, for $\Sigma \not\cong S^n$, there is no candidate geometric Dehn twist, since $\Sigma$ will not in general admit a metric with periodic geodesic flow.  Informally, one would like the non-existence of a Dehn twist to be reflected in the non-surjectivity of a map
\[
\pi_0 \Symp_{ct}(T^*\Sigma) \longrightarrow \Auteq(H\cF(T^*\Sigma))/ \langle [1] \rangle.
\]
As formulated, this is slightly problematic, since the category $\cF(T^*\Sigma)$ has too few objects (all closed objects being isomorphic up to shift).  One solution is to consider an enlarged category associated to a cotangent bundle which incorporates suitable non-compact Lagrangian submanifolds, for instance those co-inciding outside a compact subset with a fixed cotangent fibre. Such a category was invoked by Seidel in \cite{seidel:kronecker}, and related but more general versions have been considered in the work of Nadler and Zaslow \cite{NZ}.  Thus,  in the Nadler-Zaslow category $NZ(T^*S^n)$, the Dehn twist and the shift functor are indeed distinct autoequivalences, in contrast to their actions on $\cF(T^*S^n)$.

To keep technical details to a minimum, we take a slightly more circumspect approach.  Let $M_{\Sigma} = T^*\Sigma \# T^*S^n$  denote the Weinstein manifold obtained by plumbing $T^*\Sigma$ and $T^*S^n$.  Equivalently, $M_{\Sigma}$ is obtained by performing a Weinstein surgery along the Legendrian unknot which is the boundary of a cotangent fibre in $T^*\Sigma$.  Then $M_{\Sigma}$ contains two obvious closed exact Lagrangians, the components $\Sigma$ and $S^n$ of its compact core. Compactly supported symplectomorphisms of $T^*\Sigma$ act on the Fukaya category $\cF(M_{\Sigma})$ provided we quotient out by the shift $[1]$ (alternatively one could consider graded symplectomorphisms).  

\begin{rem}\label{rem:surgery-versus-nadlerzaslow}
The wrapped Fukaya category $\cW(T^*Q)$ is invariant under arbitrary exact symplectomorphisms of $T^*Q$, whilst the Nadler-Zaslow category is invariant only under compactly supported symplectomorphisms of $Q$.  The Fukaya category of the plumbing $\cF(T^*\Sigma \# T^*S^n)$ mediates between the two, being  invariant under symplectomorphisms which co-incide with the identity near the boundary of a single cotangent fibre.  It is therefore more sensitive to ``compactly supported" phenomena than the full wrapped category $\cW(T^*Q)$, even if \emph{a priori} less so than $NZ(T^*Q)$.  \end{rem}

The following is straightforward.

\begin{lem}
The twist functor $T_{\Sigma} \in \Auteq(H\cF(M_{\Sigma})) / \langle [1] \rangle$ has infinite order, in particular this quotient group is non-trivial.
\end{lem}

\begin{proof}
Given our cohomological assumptions, this follows from the usual proof of injectivity of the braid group $Br_3^{\alg}$ into the autoequivalences of the Fukaya category of the $A_2^n$-Milnor fibre. Explicitly, writing $T_{\Sigma}^i$ for the $i$-fold self-composition of $T_{\Sigma}$,  the Floer cohomology groups $HF(S^n, T_{\Sigma}^i(S^n))$ are unbounded in rank as $i\rightarrow \infty$.
\end{proof}

By contrast, the autoequivalence $T_{\Sigma}$ cannot arise from a symplectomorphism unless $\Sigma$ is homeomorphic to $S^n$.

\begin{prop} \label{prop:twist-not-geometric}
If $\pi_1(\Sigma) \neq 1$, the cyclic subgroup $\bZ$ generated by $T_{\Sigma}$ meets the image of the natural map 
\[
\rho: \pi_0\Symp_{ct}(T^*\Sigma) \longrightarrow \Auteq(H\cF(M_{\Sigma}))/\langle [1]\rangle
\] 
only in the identity.
\end{prop}

The proof will rely on passing to the universal cover, so we let $M$ be any Liouville manifold with universal cover $\pi: \tilde{M}\rightarrow M$. The following is a summary of the results proved in \cite[Section 6]{Abouzaid-ExactLag}:

\begin{thm} \label{thm:category-of-cover}
There is a wrapped Fukaya category $\cW(\tilde{M};\pi)$ which comes with a pullback functor
\[
\pi^*: \cW(M) \longrightarrow \cW(\tilde{M};\pi)
\]
which acts on objects $L$ of $\cW(M)$ by taking the total inverse image $\pi^{-1}(L) \subset \tilde{M}$ and such that the map on morphisms
\begin{equation}
  HF^{*}(L,L) \to HF^{*}(\pi^{-1}(L), \pi^{-1}(L) )
\end{equation}
agrees with the classical pullback on cohomology whenever $L \subset M$ is closed.  Moreover, deck transformations of $\pi$ act by autoequivalences of $\cW(\tilde{M};\pi)$.  \end{thm}
\begin{rem}
For infinite coverings, the category $\cW(\tilde{M};\pi)$ is not in general the same as the wrapped category of the underlying Stein manifold as defined without reference to the covering $\pi$.  For instance, if $\pi: T^*\bR \rightarrow T^*S^1$ is the universal covering, then $\cW(T^*\bR; \pi)$ is not empty -- the fibre and zero-section are non-trivial -- whereas the usual wrapped category $\cW(\bC)$ vanishes.
\end{rem}

We proceed to the proof of Proposition \ref{prop:twist-not-geometric}. Suppose $\Sigma$ is a homology sphere with non-trivial fundamental group.  Let $\iota: \tilde{\Sigma} \rightarrow \Sigma$ denote the universal covering. We fix a coefficient field $k$ of characteristic dividing the index of the covering $\iota$, interpreted as meaning that $k$ is arbitrary if this covering is infinite.   Suppose for contradiction that $T_{\Sigma}$ is geometric, induced by a compactly supported symplectomorphism $\tau$.  Let $M_{\Sigma}$ be the Weinstein manifold obtained by capping the Legendrian unknot in $D^*\Sigma$.  Considering $\tau \circ \tau$, or its inverse, and recalling the definition of the twist functor in Equation \ref{eqn:twist}, we infer that there is an exact Lagrangian submanifold $L \subset M_{\Sigma}$ representing the twisted complex 
\begin{equation} \label{Eqn:ImpossibleComplex}
\Sigma[n-1] \leftarrow \Sigma \leftarrow S^n
\end{equation}
with the leftmost arrow given by multiplication by the fundamental class $[\Sigma]$, and where $S^n$ is the core component coming from capping a fibre of $D^*\Sigma$.   (Applying shifts, we have assumed that $S^n$ is placed in degree zero and that the morphism $\Sigma \leftarrow S^n$ is in degree 1.)  Let $\pi: \tilde{M}_{\Sigma} \rightarrow M_{\Sigma} $ denote the (universal) covering induced by $\iota: \tilde{\Sigma} \rightarrow \Sigma$. The inverse image of this complex in $\cW(\tilde{M}_{\Sigma};\pi)$ is represented by
\[
\tilde{\Sigma}[n-1] \leftarrow \tilde{\Sigma} \leftarrow \pi^{-1}(S^n)
\]
in which \emph{the left differential} is obtained by pullback from the fundamental class of $\Sigma$, hence \emph{vanishes}  (trivially if $\iota$ is infinite; by our choice of characteristic for the coefficient field $k$ if $\iota$ is finite).  

By the long exact triangle for the Dehn twist, the complex $\Sigma \leftarrow S^n$ is the image of $\Sigma$ under the twist in $S^n$. Now, $\pi^{-1}(\Sigma)$ is connected, so the corresponding object in the category $\cW(\tilde{M}_{\Sigma};\pi)$ is indecomposable (considering the rank of its $HW^0$). This in turn implies that  $ \tilde{\Sigma}[n-1] \leftarrow \pi^{-1}(S^n)$ does not admit any non-trivial summand.  

\begin{lem} \label{lem:fibre-not-zerosection}
In the category $\cW(\tilde{M}_{\Sigma};\pi)$,  $\tilde{\Sigma}[n-1]$ and  $ \tilde{\Sigma} \leftarrow \pi^{-1}(S^n)$  do not lie in the same orbit under the action by deck transformations.
\end{lem}

This follows immediately from taking Floer cohomology with a cotangent fibre to one of the components of $  \pi^{-1}(S^n)$: the Floer group with  have rank $0$ in the first case, and $1$ in the other.

\begin{proof}[Proof of Proposition \ref{prop:twist-not-geometric}]
Being the preimage of a connected Lagrangian submanifold $L$ under a covering, all the components of $\pi^{-1}(L)$ are diffeomorphic and moreover related by deck transformations.  These lift to autoequivalences of  $\cW(\tilde{M}_{\Sigma};\pi)$ by Theorem \ref{thm:category-of-cover}. The analysis above showed that, up to shifts of the pieces,
\begin{equation} \label{Eqn:SplitPreimage}
\pi^{-1}(L) \cong \tilde{\Sigma} \oplus \left( \tilde{\Sigma} \leftarrow \pi^{-1}(S^n) \right)
\end{equation}
in the category $\cW(\tilde{M}_{\Sigma})$.  Each connected component $\tilde{L}_{\alpha} \subset \pi^{-1}(L)$ is itself an indecomposable object of the category, since $HW^0(\tilde{L}_{\alpha}, \tilde{L}_{\alpha}) $ has rank 1.  For the same reason, $\tilde{\Sigma}$ is indecomposable, and Lemma \ref{lem:fibre-not-zerosection} implies that $ \tilde{\Sigma} \leftarrow \pi^{-1}(S^n)$ is not isomorphic to a deck transformation image of $\tilde{\Sigma}$.  It follows that the RHS of \eqref{Eqn:SplitPreimage} is not quasi-isomorphic to a direct sum of indecomposables lying in a single orbit of the covering group, which is a  contradiction (recall that working over a field, the decomposition of an object into a sum of indecomposables is unique).  The conclusion is that the twisted complex of \eqref{Eqn:ImpossibleComplex} cannot represent an exact Lagrangian $L$, hence $\tau$ did not exist.  \end{proof}

This argument is actually a bit more general.  Suppose $\Sigma$ is a homology sphere to which the analysis underlying Proposition \ref{prop:braid-orbit} applies, for instance $\pi_1(\Sigma)$ has no non-trivial finite-dimensional representations (or the universal cover of $\Sigma$ is acyclic).  Let $\phi: T^*\Sigma \rightarrow T^*\Sigma$ be a compactly supported symplectomorphism, and write $\Phi$ for the associated element of $\Auteq(H\cF(M_{\Sigma}))$, defined with respect to an arbitrary choice of grading of $\phi$.   By the classification of exact Lagrangian submanifolds of Maslov class zero in $T^*\Sigma$, the symplectomorphism $\phi$ preserves the object $\Sigma$ up to shift.     The image $L$ of the core component $S^n \subset M_{\Sigma}$ under $\phi$ co-incides with $S^n$ outside a compact subset of $T^*\Sigma$, hence 
\begin{equation} \label{Eqn:Rank1}
HF(L, T_x^*S^n) \cong k \ \textrm{has rank} \  1.
\end{equation} Theorem \ref{thm:main} implies that $L$ is a twisted complex on $S^n$ and $\Sigma$, and \eqref{Eqn:Rank1} (together with the fact that multiplication by $[S^n]$ on $HF(S^n, T_x^*S^n)$ is trivial) implies that this complex contains a unique copy of $S^n$.  Lemma \ref{lem:first-step}, and its analogue with the roles of $Q_0$ and $Q_1$ reversed, then implies that $L$ may be represented by a complex of the shape
\begin{equation} \label{Eqn:TwoCases}
\cC^{\bullet} \leftarrow S^n \quad \textrm{or} \quad S^n \leftarrow \cC^{\bullet}
\end{equation}
with $\cC^{\bullet} \in Tw(\Sigma)$ the twisted complex $\Sigma \rightarrow \cdots \rightarrow \Sigma$ of arbitrary length and with all differentials the fundamental class $[\Sigma]$ (more explicitly, Lemma \ref{lem:first-step} implies that $V_{i\neq 0} = 0$ and that $V_0$ has dimension 1).  

The two cases of \eqref{Eqn:TwoCases} are again analogous, and we treat only the first. Place $S^n$ in degree zero, and consider the inverse image in $\cW(\tilde{M}_{\Sigma};\pi)$, represented by a complex
\begin{equation} \label{eqn:lifted-complex}
\tilde{\cC}^{\bullet} \leftarrow \pi^{-1}(S^n)
\end{equation}
where all differentials in $\tilde{\cC}^{\bullet}$ vanish, being obtained by pullback from the fundamental class of $\Sigma$. This preimage  is again quasi-isomorphic to a direct sum 
\begin{equation}
  \label{eq:direct_sum_from_homological_alg}
\pi^{-1}(L) \simeq  \tilde{\Sigma}[k(n-1)] \oplus \cdots \oplus \tilde{\Sigma}[2(n-1)] \oplus \left(  \tilde{\Sigma}[n-1] \leftarrow \pi^{-1}(S^n) \right)  
\end{equation}
Lemma \ref{lem:fibre-not-zerosection} yields a contradiction unless 
\begin{equation} \label{eqn:only-way-out}
L \simeq (\Sigma \leftarrow S^n) \quad \textrm{or} \quad L \simeq S^n
\end{equation}
by following the same steps as in the proof of Proposition  \ref{prop:twist-not-geometric}.
In the first case, replacing $\phi$ by an iterate if necessary, we would obtain a new Lagrangian submanifold representing the complex $\Sigma \leftarrow \Sigma \leftarrow S^n$, which contradicts \eqref{eqn:only-way-out}. Therefore $L\simeq S^n$, and moreover one can easily check  that $\phi$ acts trivially on the Floer cohomology algebra generated by $\Sigma$ and $S^n$.   In other words, we conclude:

\begin{cor}
Let $\Sigma$ be a homology sphere with $\pi_1(\Sigma) \neq 1$. Suppose $\pi_1(\Sigma)$ admits no non-trivial finite-dimensional representation.  Any $\phi \in \Symp_{ct}(T^*\Sigma)$ acts trivially modulo shift on $H^*(\cF(M_{\Sigma}))$. 
\end{cor}

The symplectomorphism $\phi$ plays a rather minor role above, which leads to a slightly more general formulation of the conclusion.  Inside $T^*\Sigma$, consider an exact Lagrangian submanifold $K$ which is a \emph{filling} of the Legendrian unknot $\Lambda$. That is, inside the unit disk cotangent bundle, we consider a properly embedded Lagrangian submanifold with $\Lambda$ as Legendrian boundary.  

\begin{cor}
If $\Sigma$ is a homology sphere with $\pi_1(\Sigma) \neq 1$, and if $\pi_1(\Sigma)$ admits no non-trivial finite-dimensional representations, then any exact Maslov zero Lagrangian filling $K$ of $\Lambda$  is a homology ball.
\end{cor}

\begin{proof}
Adding a Weinstein handle along the Legendrian unknot yields an associated closed Lagrangian submanifold $\bar{L} \subset M_{\Sigma}$, with $M_{\Sigma} = T^*\Sigma \# T^*S^n$ the plumbing considered previously.  The analysis surrounding Equation \ref{eqn:only-way-out} implies that
\[
\bar{K} \simeq S^n \quad \textrm{or} \quad \bar{K} \simeq (\Sigma \leftarrow S^n) \quad \textrm{or} \quad \bar{K} \simeq (S^n \leftarrow \Sigma) \ \in \ H\cF(M_{\Sigma}).
\]
These three objects are quasi-represented by the core component $S^n$ and its two possible Lagrange surgeries with the other component $\Sigma$, respectively.  It follows that $K$ is Floer cohomologically indistiguishable from at least one of the fibre $T_x^*\Sigma$ or one of its surgeries with the $0$-section, in particular has the same cohomology groups as one of these spaces. Since the complement of a disk in a homology sphere is a homology ball, the result follows.
\end{proof}

\begin{rem} 
Some care should be exercised in extending the results here from $\bZ$-homology spheres to $k$-homology spheres. Although the basic classification of twisted complexes carries over, cf. Remark \ref{rem:k-homology-sphere}, there is potentially conflict between the characteristic of $k$ for which one obtains a homology sphere and the characteristic required in the vanishing argument for differentials in the twisted complex in the universal cover. For instance, $\bR\bP^3$ is a homology sphere except in characteristic 2; however, the vanishing argument after Equation \ref{eqn:lifted-complex} would precisely require characteristic 2, i.e. a divisor of $|\pi_1(\bR\bP^3)|$, and $T^*\bR\bP^3$ \emph{does} admit a non-trivial autoequivalence (Dehn twist).
\end{rem}

\subsubsection{Stablisations and symplectomorphisms}
\label{sec:Kontsevich_stab}

Consider the operation which assigns to a Liouville manifold $M$ its product with the unit disc $D^{2}$ together with the function
\begin{align}
 W \co M \times D^{2} & \to D^{2} \\
(m,z) & \mapsto z^{2}.
\end{align}

Using the techniques developed in \cite{seidel-book}, one may construct  a well defined Fukaya category $\sF(M \times D^{2}, W)$ with objects being exact Lagrangian branes with boundary contained in $M \times \{1\}$.    The study of these categories is not fully developed, but one expects the following conjecture to hold:
\begin{conj} \label{conj:covers}
 For each cover $\tilde{M}$ of $M$ there is a commutative diagram
 \begin{equation}
   \xymatrix{  \sF(\tilde{M}) \ar[r] & \sF(  \tilde{M} \times D^{2}, W)    \\
 \sF(M) \ar[r] \ar[u] & \sF(  M \times D^{2}, W). \ar[u]  }
 \end{equation}
where the horizontal arrows are equivalences, and the vertical arrows are pullback functors which assign to a Lagrangian its inverse image.
\end{conj}

Assuming this conjecture, we observe that whenever $M$ is the plumbing of $T^{*} \Sigma$ and $T^{*} S^{n}$, with $\Sigma$ a homology sphere with non-trivial fundamental group, there can be no geometric automorphism of $  \tilde{M} \times D^{2} $ realising the Dehn twist in $ \sF(  \tilde{M} \times D^{2}, W) $, cf. Remark \ref{rem:stabilise}.    Indeed, most arguments given in the proof of Proposition \ref{prop:twist-not-geometric} use only algebra, and hence would apply just as well after stabilisation.  The only geometric part of the argument appears in our analysis of Equation \eqref{Eqn:SplitPreimage}, where we use the fact that the inverse images of a Lagrangian under a covering map are disjoint, hence Floer cohomologically orthogonal, and related by deck transformations.  While we are not proposing a precise definition of the class of symplectomorphisms of $  \tilde{M} \times D^2 $ which give rise to automorphisms of the category $  \sF(  \tilde{M} \times D^2, W)$, any sufficiently natural definition will still have the property that the deck transformations act, and that is all that was needed.

\subsection{Truncated polynomial rings}
\label{sec:trunc-polyn-rings}

The construction of the geometric Dehn twist $\tau_{L}$ in a Lagrangian sphere $L\subset M$ relies on the existence of a metric on $S^n$ with periodic geodesic flow.  Similar metrics exist on all the real, complex and quaternionic projective spaces, and the Cayley projective plane. In general, if a manifold $Z$ admits a metric for which all geodesics are closed, they are necessarily of the same length, and the geodesic flow is periodic.  Such manifolds are called \emph{Zoll} manifolds, and have been the subject of much classical investigation.  Bott \cite{Bott} proved, using Morse theoretic analysis of spaces of geodesics, that if a manifold $Z$ admits a Zoll metric then $H^*(Z)$ is a truncated polynomial ring (and $\pi_1(Z)$ is either trivial or of order 2).  

\begin{conj} \label{conj:truncated}
Given a simply-connected $n$-manifold $Q$, let $M_Q$ denote the plumbing $T^*Q \# T^*S^n$ given by adding a Weinstein handle to the Legendrian unknot in a fibre of $T^*Q$.  If the map
\[
\rho: \ \pi_0\Symp_{ct}(T^*Q) \longrightarrow \Auteq(H\cF(M_Q))/\langle [1] \rangle
\]
is non-trivial, then $H^*(Q;\bZ)$ is a truncated polynomial ring.
\end{conj}

A cleaner formulation, avoiding the auxiliary manifold $M_Q$, invokes the Nadler-Zaslow category $NZ(T^*Q)$ in place of $\cF(M_Q)$, cf. Remark \ref{rem:surgery-versus-nadlerzaslow}.  We note that classical work of Adem gives strong constraints on the manifolds whose cohomology is a truncated polynomial algebra; writing $H^*(Q) \cong \bZ[x]/\langle x^d\rangle$, necessarily the degree $|x|$ of $x$ is a power of 2, and if $|x| > 4$ then $x^3=0$. The only known simply-connected manifolds with truncated polynomial cohomology are homeomorphic to the compact rank one symmetric spaces.  The rest of this section is devoted to providing evidence for Conjecture \ref{conj:truncated}, proving a partial analogue in two non-trivial cases.

\begin{prop} \label{prop:evidence}
In the situation of Conjecture \ref{conj:truncated}, suppose moreover that all higher products
\[
\mu^p: H^{i_1}(Q;k) \otimes \cdots \otimes H^{i_p}(Q;k) \rightarrow H^{i_1 + \cdots + i_p + 2-p}(Q;k)
\] 
with
\[ p \geq 2 \quad \textrm{and} \  0 < i_j < \sum i_j +2-p < n=dim_{\bR}(Q) 
\]
vanish, i.e. vanishing holds whenever the output lies in cohomological degree $< n$. If for $\phi \in \Symp_{ct}(T^*Q)$ the autoequivalence $\Phi = \rho(\phi)$ acts non-trivially on objects of $\cF(M_Q)$, then 
\[
\sum_{0<i<n} \, rk_k \,H^i(Q;k) \in \{0,1\}.
\]
\end{prop}

Note that the hypothesis applies to an $(n-1)$-connected $2n$-manifold $Q$, implying that $H^n(Q;k)$ has rank $\leq 1$. The hypothesis is also satisfied if $Q$ has \emph{formal} cohomology ring, and if the only non-trivial cup-products are those forced by Poincar\'e duality (i.e. those into the top cohomological degree); for instance, if $Q$ is a product of spheres, or a connect sum of such. Thus, Proposition \ref{prop:evidence} implies Proposition \ref{prop:partial_result_truncated} from the Introduction.

We establish some preliminary lemmas relevant for Proposition \ref{prop:evidence}. Suppose an exotic autoequivalence $\Phi$ arising from a compactly supported symplectomorphism $\phi$ does exist, and apply it to the core component $S^n$.  By Theorem \ref{thm:main}, the image is a twisted complex on $S^n$ and $Q$.  Compact support of $\phi$ implies that $\Phi(Q) \simeq Q[j]$ is quasi-isomorphic to some shift of $Q$, by uniqueness of exact Lagrangian submanifolds of Maslov class zero in simply-connected cotangent bundles \cite{FSS, nadler}.  In particular, the group 
\begin{equation} \label{eqn:rank-one}
HF(\Phi(S^n), Q) \ \cong \ HF(S^n, Q[j]) \ \cong \ k
\end{equation}
has rank 1 over the field.   Equation \eqref{eqn:rank-one} implies that the twisted complex representing $\Phi(S^n)$ contains a unique copy of $S^n$.  Although $Q$ is not a sphere, the conclusion of Lemma \ref{lem:first-step} still applies to twisted complexes over $Q$ and $S^n$, since that argument only used the fact that no identity morphisms appeared in the differentials of the twisted complex.  As in the previous subsection, (the analogue of) Lemma \ref{lem:first-step} then implies that $\Phi(S^n) $ is equivalent to a twisted complex of the form
\[
\left( \cC^{\bullet}  \leftarrow S^n \right) \quad \textrm{or} \quad  \ \left( S^n \leftarrow \cC^{\bullet} \right)
\]
with $\cC^{\bullet}$ only involving the component $Q$. Since we are assuming that $\Phi$ acts non-trivially on objects modulo shift, and since $\Phi$ fixes $Q$ modulo shift, the complex $\cC^{\bullet}$ is not empty.   Replacing $\phi$ by $\phi^{-1}$ and $\Phi$ by $\Phi^{-1}$ if necessary, we will henceforth concentrate on the case $\Phi(S^n) \simeq (\cC^{\bullet}\leftarrow S^n)$. 

\begin{lem} \label{Lem:Rank2}
$HF(\cC^{\bullet}, Q)$ has rank 2.
\end{lem}

\begin{proof} $\cC^{\bullet}$ lies in an exact triangle with $S^n$ and $\Phi(S^n)$, which have  Floer homology of rank 1 with $Q$, in the latter case by \eqref{eqn:rank-one}. Exactness implies that $HF(\cC^{\bullet}, Q)$ has rank $0$ or $2$. However, since $\cC^{\bullet}$ is a twisted complex on $Q$, it can only be orthogonal to $Q$ if it is identically zero, which we have excluded by the hypothesis that $\Phi$ acts non-trivially on objects.
\end{proof}

Let $V_i$ denote the vector spaces underlying the twisted complex $\cC^{\bullet}$, so 
\[
\Phi(S^n)\  \cong \quad V_0 \otimes Q \leftarrow V_1 \otimes Q \leftarrow \cdots \leftarrow V_N\otimes Q \leftarrow S^n
\]
where by assumption $V_0 \neq 0$ and $V_N \neq 0$. 

\begin{lem}
Either $Q$ is a $k$-homology sphere, or  $N>0$.
\end{lem}

\begin{proof}
If $\cC^{\bullet} = V \otimes Q$, then $HF(\cC^{\bullet},Q) \cong V^{\vee} \otimes H^*(Q)$. This can only have rank 2 if $V\cong k$ and $Q$ is a homology sphere.
\end{proof}

Since we know that homology spheres admit categorical Dehn twists, we will henceforth exclude this case, and suppose that $N>0$.

\begin{lem}
The outermost vector spaces in $\cC^{\bullet}$ are each of rank 1, i.e.
\begin{equation} \label{Eqn:RankOneAtEnds}
dim(V_0) = 1 = dim(V_N),
\end{equation}
and we have an isomorphism
\begin{equation}
  Hom(V_N, H^0(Q)) \oplus Hom(V_0, H^n(Q)) \cong  HF(\cC^{\bullet},Q) 
\end{equation}
induced by the inclusion of the left hand side in the chain complex of morphisms from $\cC^{\bullet} $ to $Q$.
\end{lem}
\begin{proof}
Any element of $Hom(V_0, H^n(Q))$ defines a morphism  $\cC^{\bullet} \rightarrow Q$ as follows:
 \begin{equation}
\xymatrix{
V_0 \otimes Q \ar[d]_{[Q]} & \ar[l] \cdots & V_N \otimes Q \ar[l] \\ Q && 
}
\end{equation}
which is closed, since $V_0$ is terminal in $\cC^{\bullet}$, but cannot be exact, from the usual assumption that the twisted complex $\cC^{\bullet}$ contains no differentials represented by $e_Q$ (and using strict unitality to exclude higher-order $A_{\infty}$-products).  Analogously, $Hom(V_N, H^0(Q))$ embeds into $HF(\cC^{\bullet},Q)$: the morphism
 \begin{equation}
\xymatrix{
V_0 \otimes Q & \ar[l] \cdots & V_N \otimes Q \ar[l] \ar[d]_{e_Q} \\ && Q 
}
\end{equation}
is exact, since the $V_N$-copy of $Q$ is initial, and cannot be closed since the differential $d_{\cC^{\bullet}}$  contains no $e_Q$-terms.  It is moreover easy to see that distinct morphisms arising in this way cannot be cohomologous if the corresponding elements of $V_0$ respectively $V_N$ are not linearly dependent, so 
\[
Hom(V_0, H^{n}(Q)) \oplus Hom(V_N, H^0(Q)) \subset HF(\cC^{\bullet}, Q). 
\]  Lemma \ref{Lem:Rank2} now implies that $dim(V_0) + dim(V_N) \leq 2$, which in turn implies \eqref{Eqn:RankOneAtEnds}.
\end{proof}

\begin{proof}[Proof of Proposition \ref{prop:evidence}]
Let $V = \oplus_i V_i$ denote the sum of the vector spaces in the twisted complex defining $\cC^{\bullet}$.  Lemma \ref{Lem:Rank2} implies that under the given assumptions on $Q$, namely vanishing of non-trivial cup-products and higher products into $H^{<n}(Q;k)$, the complex computing $HF(\cC^{\bullet},Q)$ includes an acylic complex
\begin{equation} \label{eq:decomponse_differential}
\xymatrix{
& Hom( V_{0}, H^{0<i<n}(Q))   &  \\
 Hom( \oplus_{i=0}^{N-1}V, H^0(Q)) \ar[r] \ar[ur] \ar@/^1.5pc/[rr] & Hom(\oplus_{i=1}^{N-1} V_{i}, H^{0<i<n}(Q)) \ar[r] & Hom( \oplus_{i=1}^{N}V, H^n(Q))  \\
& Hom( V_{N}, H^{0<i<n}(Q)) \ar[ur] & 
}
\end{equation}
with an additional differential from the first to the last term, as indicated.  Acyclicity implies that the differential surjects onto the top term in the middle column. Let $U^0 $ denote the kernel of the composition of the differential with the projection onto this summand; this is a proper subspace of $Hom( \oplus_{i=0}^{N-1}V, H^0(Q))$ whose rank is equal to $\dim(V)-1 - \sum_{0<i<n} b^i(Q)$, with $b^i(Q) = rk_k H^i(Q;k)$ the $i$-th Betti number.  Dually, let $U^{N}$ denote the quotient of  $Hom( \oplus_{i=1}^{N}V, H^n(Q))$ by the image of the restriction of the differential to the last term in the middle column; this vector space also has rank $\dim(V)-1 - \sum_{0<i<n} b^i(Q)  $. 

Acyclicity of Equation (\ref{eq:decomponse_differential}) is now equivalent to the fact that the complex
\vspace{.1in}
\begin{equation}
  \xymatrix{ U^0 \ar[r] \ar@/^1.5pc/[rr] &  Hom(\oplus_{i=1}^{N-1} V_{i} , H^{0<i<n}(Q))  \ar[r] &  U^{N}}
\end{equation}
is exact.  Exactness in the middle now implies that
\begin{align}
  (dim(V)-2) \cdot \sum_{0<i<n} b^i(Q) & \leq  2 \left(dim(V) - 1 -  \sum_{0<i<n} b^i(Q) \right) \\
  dim(V) \cdot \left(\sum_{0<i<n} b^i(Q)-2 \right) & \leq  -2
\end{align}
This is  impossible whenever $\sum_{0<i<n} b^i(Q) > 1$, because the left hand side is then non-negative, which implies the required conclusion.
\end{proof}

It seems reasonable to call a compactly supported symplectomorphism of $T^*Q$ yielding an autoequivalence of $T^*Q$ which acts non-trivially on objects of $\cF(M_Q)$ a ``categorical Dehn twist".  

\begin{cor}
Let $Q$ be a simply-connected closed 4-manifold. If $T^*Q$ admits a categorical Dehn twist then $Q$ is homotopy equivalent to $S^4$ or $\bC\bP^2$.
\end{cor}

\begin{proof}
Proposition \ref{prop:evidence} implies that $H^2(Q;k)$ is trivial or of rank 1. Since this holds for any co-efficient field $k$, the result is true integrally, and then the universal coefficient theorem implies that $H^*(Q;\bZ)$ is torsion-free.  $Q$ is therefore a homotopy sphere or a homotopy complex projective plane, by (Milnor's refinement of) Whitehead's theorem \cite{Milnor}.  \end{proof}

\section{Fibres generate the wrapped Fukaya category}

\label{sec:generation_wrapped_category}
Let $Q_0$ and $Q_1$ be two smooth closed oriented manifolds of the same dimension $n$.  From \cite{fibres-generate}, we know that a cotangent fibre generates the wrapped Fukaya category of each of $T^{*}Q_{0}$ and $T^{*}Q_1$.  The goal of this section is to prove the analogous result for the plumbing of the two cotangent bundles.  To state the result, note that one can always arrange the plumbing parameters so that the plumbing region is disjoint from a pair of cotangent fibres  $\Tq[i]$.  In particular, these give rise to exact Lagrangians in $T^*Q_0 \# T^*Q_1  $ which we denote $\L[i]$, and hence to objects of its wrapped Fukaya category.
\begin{thm} \label{thm:wrapped_category_generation}
Any pair of cotangent fibres $\L[0]$ and $\L[1]$ which are disjoint from the plumbing region,  generate the subcategory of the wrapped Fukaya category of $T^*Q_0 \# T^*Q_1$ consisting of closed Lagrangians.
\end{thm}

The weaker statement that these Lagrangians split-generate the wrapped Fukaya category should follow from work of Bourgeois, Ekholm, and Eliashberg \cite{BEE}, together with the generation criterion of \cite{generate}.  Unfortunately, the bridge between the Symplectic Field theoretic and Hamiltonian approaches to studying the Fukaya category is not fully in place, and the technicalities of building such a bridge would go beyond the scope of this paper.

Instead, we give an alternative approach in two parts.  First, we consider an additional Lagrangian in $ T^*Q_0 \# T^*Q_1 $, coming from the diagonal Lagrangian in a copy of $\bC^{n} = \bR^n \oplus \i  \bR^n$ containing the model plumbing region. Using Seidel's double covering trick \cite[Section 18]{seidel-book}, and the Viterbo restriction functor \cite{AS}, we show that this Lagrangian, together with the two cotangent fibres appearing in the statement of Theorem \ref{thm:wrapped_category_generation}, generates the Fukaya category.  Then we use an explicit geometric argument, together with the restriction functor again, to show that the diagonal can be removed from our generating collection.

 \subsection{Embedding the plumbing in a Lefschetz fibration}

\label{sec:embedd-plumb-lefsch}
The starting point is to consider the manifold $Q$ obtained by removing a small ball from each of $Q_0$ and $Q_1$, and gluing the resulting boundary spheres via their natural identifications with $S^{n-1}$. Clearly $Q$ is diffeomorphic to the usual connect sum of the $Q_i$, and contains a separating sphere $V$.

By work of \cite[Section 3]{FSS}, we know that any Morse function on $Q$ extends to a Lefschetz fibration $\pi^{0}$ on an exact manifold with corners $E^{0}$, over the disc of radius $1/2$, containing $Q$ as an exact Lagrangian.  We shall choose a Morse function which maps the region coming from $Q_{0}$ to the negative reals,  the region coming from $Q_{1}$ to the positive reals, and which has no critical values near the origin.  The inverse image of $0$ is therefore the separating $n-1$ dimensional sphere $V \subset Q$, which gives rise to a Lagrangian sphere in the fibre of $\pi^0$.  Note that this is a \emph{framed} Lagrangian in the sense of Seidel, carrying a canonical identification with the sphere up to the action of $O(n)$.

By perturbing  the Lefschetz fibration, we may arrange that there are no critical points along the $y$-axis, that the fibration is trivial in a neighbourhood thereof, and that the only critical points along the $x$-axis lie on $Q$ (see Figure \ref{fig:lefschetz-fibration}).  We may construct a new exact manifold $E$ equipped with a Lefschetz fibration $\pi$ over the disc of radius $1$, with the same fibre as $\pi^0$, by adding exactly one new critical point to $\pi^0$, whose value lies on the segment between $(0,1/2)$ and $(0,1)$, and whose vanishing cycle agrees with $V$ by parallel transport along the $y$-axis from this new critical point to the origin.  Let $T$ denote the Lagrangian disc (thimble) obtained by parallel transport of $V$.  Note that $Q \cap T = V$.

\begin{lem}
  There are exact Lagrangian embeddings of $Q_0$ and $Q_1$ in $E$ which intersect transversely at one point, whose image is contained in a neighbourhood of $Q \cup T$, and which agree with $Q$ away from $T$.
\end{lem}
\begin{proof}
 We construct $Q_0$ as the union of two pieces:  First, consider the Lagrangian obtained by taking the intersection of $Q$ with the inverse image of $(-1/2,-\epsilon]$.  This is a Lagrangian diffeomorphic to the complement of a ball in $Q_0$, with the boundary sphere mapping to $-\epsilon$.  Next, choose any smooth path contained in the neighbourhood of the $y$-axis connecting $(-\epsilon,0)$ to the new critical value of $\pi$, which only intersects the $y$-axis at this critical value and agrees with a horizontal line in a neighbourhood thereof (see Figure \ref{fig:lefschetz-fibration}).   Having assumed that the fibration $\pi^{0}$ is trivial in a neighbourhood of the $y$-axis, this is also a thimble with the correct vanishing cycle (and hence topologically a ball).  The union of these two Lagrangians is obviously homeomorphic to $Q_0$, but it is actually diffeomorphic because the framing on $V$ was obtained by using its identification with the unit sphere in the tangent space of a point on $Q_0$.
 \begin{figure}
   \centering
   \includegraphics{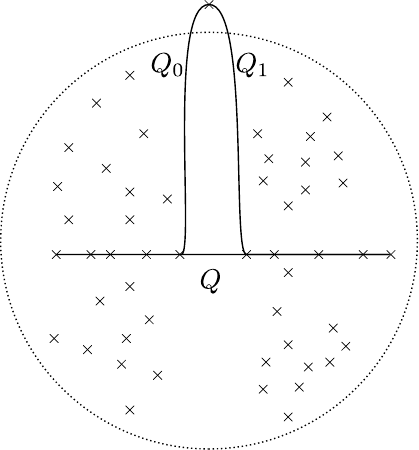}
   \caption{ }
   \label{fig:lefschetz-fibration}
 \end{figure}
We construct $Q_1$ symmetrically so that it lies in the positive $x$ half-plane.  Note that these two Lagrangians only meet at the critical point of $\pi$ which lies on the $y$-axis, proving the first part of the Lemma.  The remaining claims are obvious from the construction.
\end{proof}

Any neighbourhood of a pair of exact Lagrangians which meet transversely at one point is symplectomorphic to a neighbourhood in their plumbing.  Choosing this neighbourhood to be small enough, we may assume that the only critical points of the Lefschetz fibration that it includes are those which lie on either $Q_0$ or $Q_1$.

\subsection{General results}

\label{sec:dehn-twists-branched}

The results of this section apply beyond the special case we are considering.  Recall that a thimble in the total space of a Lefschetz fibration is a Lagrangian ball obtained by transporting a vanishing cycle from a distinguished boundary point (say $(0,1)$ in our case)  to a critical point along an embedded arc in the disc.  A basis of thimbles is a collection of such balls, one for each critical point, which are obtained by arcs which do not intersect in the interior.  We shall fix such a basis $\Delta_{1}, \ldots, \Delta_{m}$.

The double covering trick of Seidel embeds $E$ as a Liouville submanifold of a manifold $\tilde{E}$,  such that there are Lagrangian spheres $\tilde{\Delta}_{1}, \ldots, \tilde{\Delta}_{m}$ in $\tilde{E}$ whose intersections with $E$ agree with $\Delta_{1}, \ldots, \Delta_{m}$.    The construction is arranged in such a way that the composition of the Dehn twists along the spheres $ \tilde{\Delta}_{i} $ is Hamiltonian isotopic to a symplectomorphism which maps every closed Lagrangian in $E$ to a region which is disjoint from $E$.

Using the correspondence between algebraic and geometric twists, Seidel concludes:
\begin{lem}[Lemma 18.15 \cite{seidel-book}] \label{lem:seidel-generation}
Let $L$ be a closed Lagrangian brane in $E$.   There exists (i) an object $C$ in the subcategory of $\cF(\tilde{E})$ generated by $ \tilde{\Delta}_{i} $ (ii) a morphism from $C$ to $L$ and (iii) a Lagrangian $L_{-}$ which is disjoint from $E$, such that we have a quasi-isomorphism
\begin{equation}
  \Cone(C \to L) \cong L_{-}
\end{equation}
in the derived Fukaya category of $\tilde{E}$.
\end{lem}
\begin{rem}
The reader familiar with Seidel's book may benefit from being reminded that this result, because it appeals to the Fukaya category of $\tilde{E}$ rather than to that of a Lefschetz fibration, does not require the imposition of any assumption on the characteristic of the ground field.
\end{rem}

In order to derive a conclusion about wrapped Fukaya categories, we shall invoke a general result about such categories.  Let $W$ be an exact symplectic manifold whose Liouville form $\lambda$ defines a contact structure on the boundary, and assume that $W^{in} \subset W$ is a codimension $0$ submanifold which is also Liouville with respect to the restriction of $\lambda$.  The wrapped Fukaya category $\cW(W)$ of $W$ has, as objects, exact Lagrangians whose primitive is locally constant near the boundary, equipped with additional data to obtain integral gradings and oriented moduli spaces of holomorphic discs.  Consider a subcategory $\sL(W)$ of $\cW(W)$ consisting of Lagrangians $L$ satisfying the following additional property
\begin{equation}
  \parbox{30em}{$L$ admits a primitive for $\lambda$ which is constant on $\partial W^{in} \cap L$.}
\end{equation}
The following result was proved in Section 5 of \cite{AS}:
\begin{prop}[Viterbo restriction functor] \label{prop:restriction-functor}
There exists an $A_{\infty}$ functor $\sL(W) \to \Wrap(W^{in}) $ which assigns to each Lagrangian its intersection with $W^{in}$.  \noproof
\end{prop}

This result has a fairly straightforward consequence which forms the starting point of investigations of the wrapped Fukaya categories of Lefschetz fibrations undertaken by the first author with Seidel:
\begin{prop} \label{prop:wrapped_fukaya_Lefschtez}
The Fukaya category of closed Lagrangians in $E$ lies in the subcategory of $\Wrap(E)$ generated by  ${\Delta}_{1}, \ldots, {\Delta}_{m}$.
\end{prop}
\begin{proof} Let $L$ be a closed exact Lagrangian brane in $E$. 
We shall apply the restriction functor, Proposition \ref{prop:restriction-functor}, from the subcategory of $ \cF(\tilde{E}) $ whose objects are $ \tilde{\Delta}_{1}, \ldots, \tilde{\Delta}_{m} $, $L$ and $L_{-}$ to the wrapped Fukaya category of $E$.  The main subtlety  is that such a functor is only defined under the assumption that each Lagrangian is  exact and admits a primitive which is constant on its intersection with $\partial E$.  This is obvious for $L$ and $L_{-}$, and follows for $\tilde{ \Delta}_{i} $  because the boundary is connected and Legendrian for an appropriate choice of contact form on the boundary of $E$.

At the level of objects, the restriction functor assigns to each Lagrangian its intersection with $E$, so we obtain that $L_{-}$ goes to $0$, $L$ to itself, and $  \tilde{\Delta}_{i} $  to $\Delta_{i}  $.  In particular, Lemma \ref{lem:seidel-generation} implies that the image of $ \Cone(C \to L)  $ vanishes in the wrapped Fukaya category of $E$.  Since the image of $C$ lies in the category generated by the basis of thimbles, we conclude that so does $L$. 
\end{proof}

\subsection{Generators for the wrapped Fukaya category of the plumbing}

\label{sec:accel-funct-gener}

We now return to the specific Lefschetz fibration we were studying, taking the origin as basepoint.  We choose thimbles with the following property:
\begin{equation}
  \parbox{30em}{for $j=0,1$, either $\Delta_{i}$ is disjoint from $Q_j$, or it intersects $Q_j$ transversely at one point.}
\end{equation}
The thimbles disjoint from $Q_0 \cup Q_1$ correspond to critical points which do not lie on $ Q_0 \cup Q_1 $. These may essentially be discarded, by the discussion at the end of Section \ref{sec:embedd-plumb-lefsch} (they may be connected to the origin via arcs which follow the boundary of the disk of Figure \ref{fig:lefschetz-fibration} into the lower half of that picture, and which therefore do not intersect the projections of $Q_0$ and $Q_1$).  The other thimbles are those arising from critical points which lie on $Q_0$ and $Q_1$, so we cannot avoid the existence of an intersection.    In this second situation, there are three different possibilities:  First, we have the  critical points which lie on $Q_0$.  As the associated thimbles intersect $Q_0$ in one point, we may find a Weinstein neighbourhood of $Q_0$ which intersects the thimble in a (disc) cotangent fibre.  Next, there is the symmetric case replacing $Q_0$ by $Q_1$.  Finally, there is the distinguished critical point which lies on the $y$-axis.  After translation in the base, we consider a local model Lefschetz fibration
\begin{align}
  \bC^{n} & \to \bC \\
  (z_1, \ldots, z_n) & \mapsto \sum z_{i}^{2}.
\end{align}
Recall that we have arranged for $Q_0$ and $Q_1$ to project to paths which are horizontal near the critical point on the $y$-axis, so that after translation they agree with the $x$-axis.  In particular, in this local model, they must come from the Lagrangians $\bR^{n}$ and $\i  \bR^{n}$.  Moreover, we may choose the thimble associated to this critical point to project to the positive $y$-axis in this model, which then forces it to agree with $(1+\i  ) \bR^{n}$.  Using the fact that a Weinstein neighbourhood of the union of $Q_0$ and $Q_1$ is modelled after the plumbing, we conclude:
\begin{lem}
  There exists an exact Liouville embedding of $D^*Q_0 \# D^*Q_1 $ into $E$ which intersects all thimbles either in a cotangent fibre to $Q_i$ or in the diagonal in the local model of Equation \ref{eqn:model-region}. \noproof
\end{lem}
\begin{cor} \label{cor:generation_with_diagonal}
The  Fukaya category of closed Lagrangians in  $T^*Q_0 \# T^*Q_1 $  lies in the subcategory of the wrapped category which is generated by a collection of three objects:  A cotangent fibre to each Lagrangian $Q_i$, and the intersection of the diagonal with the local model of the plumbing region.
\end{cor}
\begin{proof}
  Apply the restriction functor from $E$ to $ D^*Q_0 \# D^*Q_1  $, together with Proposition \ref{prop:wrapped_fukaya_Lefschtez}.
\end{proof}

\subsection{Expressing the diagonal in terms of cotangent fibres}
In order to prove Theorem \ref{thm:wrapped_category_generation}, we shall show that the additional generator given in Corollary \ref{cor:generation_with_diagonal} is in fact redundant.  The proof will use a geometric argument, together with the restriction functor a few more times.

We shall start by considering a point $q \in \bR^{n}$  and the Lagrangians $q + \i  \bR^{n}$ and $\i  q +  \bR^{n}$.  Assuming that the model region in Equation \eqref{eqn:model-region} is sufficiently large so as to contain the point $ q +  \i q$, we respectively obtain Lagrangians  $L_{q}$ and $L_{p}$ in the plumbing, which intersect.  Using Polterovich's surgery construction \cite{Polterovich}, we may construct two other Lagrangians which agree with $L_{q} \cup L_{p}$  away from the intersection point. We recall a convenient model for these Lagrange surgeries.  Working in a suitable Darboux ball, it is sufficient to define the surgeries at the origin of the union $\bR^n \cup \i \bR^n \subset \bC^n$.   Let $\gamma: \bR \rightarrow \bC$ be a smooth embedded path satisfying the following conditions:
\begin{enumerate}
\item for $t \ll 0$, $\gamma$ is a linear parametrization of the negative imaginary axis;
\item for $t \gg 0$, $\gamma$ is a linear parametrization of the positive real axis;
\item im$(\gamma) \subset \bC$ lies in the upper left quadrant.
\end{enumerate}
Associated to such a path $\gamma$ we have the Lagrangian submanifold
\begin{equation} \label{Eqn:LagSurgery}
V_{\gamma} \ = \ \bigcup_{t\in \im(\gamma)} t.S^{n-1} \ \subset \ \bC^n
\end{equation}
where $S^{n-1} \subset \bR^n \subset \bC^n$ is the unit sphere; this co-incides outside a compact set with $\bR^n \cup\i  \bR^n$.  Any two such submanifolds are Lagrangian isotopic via a Lagrangian isotopy with compact support.  $V_{\gamma}$ defines one of the two surgeries. The other surgery is obtained similarly, but starting with a path $\gamma'$ with the same asymptotic conditions, but now passing below the origin, so im$(\gamma')$ is contained in the union of the three quadrants in $\bC$ whose interiors do not meet $\gamma$.

\begin{lem}
In the wrapped  Fukaya category of the plumbing, the surgeries are quasi-isomorphic to cones
\begin{equation} \label{eq:cones_discs}
 \Cone(  L_{p} \to L_{q}) \textrm{ and } \, \Cone(L_{q} \to  L_{p}).
\end{equation}
\end{lem}
\begin{proof}
Consider the Liouville domain obtained by attaching two Lagrangian handles to the plumbing along $\partial L_{p}$ and $\partial L_{q}$.  The Lagrangian discs $L_p$ and $L_q$ extend to Lagrangian spheres which we denote $V_p$ and $V_q$.  Let $\tau_{V_{p}} (V_{q}) $ and $\tau_{V_{q}} (V_{p}) $ denote the Lagrangian spheres obtained by applying the indicated Dehn twists.  Seidel's long exact sequence \cite{Seidel:LES} asserts that we have quasi-isomorphisms
\begin{equation}
  \tau_{V_{p}} (V_{q})  \cong \Cone(  V_{p} \to V_{q} ) \textrm{ and }  \tau_{V_{q}} (V_{p})  \cong \Cone(  V_{q} \to V_{p} ).
\end{equation}
Applying the restriction functor, we see that $V_{p}$ and $V_{q}$ respectively map to $L_{p}$ and $L_{q}$, while the two twists map to the two possible surgeries.
\end{proof}
\begin{rem}
For the purposes of this paper, we do not need to analyse the morphism appearing in Equation \eqref{eq:cones_discs}.  However, it follows easily from the construction of \cite{AS} that in both cases, the morphism corresponds to the generator of Floer cohomology arising  from the intersection point in the local model.  By increasing the norm of $q$, we may isotope the Lagrangians $L_{p}$ and $L_{q}$ to Lagrangian balls which are disjoint.  By the analysis in Section \ref{sec:non-posit-masl}, there is a ``short"  Reeb chord (one entirely contained in the local model) between these Lagrangians in only one direction, cf. Figure \ref{fig:dark_side_moon}.  Using a continuation homomorphism in Floer cohomology, we conclude that the connecting homomorphism in the cone defined by $L_{p} \to L_{q}$ in fact vanishes.  From that, we conclude that the surgery is equivalent in the wrapped Fukaya category to the disjoint union of $L_p$ and $L_q$, hence one obtains an  easy example of pairs of Lagrangians which are isomorphic non-zero objects in the Fukaya category but not Lagrangian isotopic, indeed which have different topology.  More generally, one can start with any pair of exact Lagrangians which do not intersect, and consider the shortest Reeb chord between them (there is a unique one generically).  There is a corresponding Hamiltonian isotopy which makes the pair intersect at exactly one point.  Using the surgery formula of \cite{FO3}, one can prove again that  one of the two surgeries is equivalent to the disjoint union in the wrapped Fukaya category, so this phenomenon is rather general.\end{rem}

We shall now show that the cone of $ L_{q} \to L_{p}$ is equivalent, in the wrapped Fukaya category of the plumbing, to the diagonal Lagrangian.  The proof will use, yet again, the restriction functor, as well as the invariance of the quasi-isomorphism classes of objects under Hamiltonian isotopies.  The key geometric observation is the following:

\begin{lem} \label{lem:surgery__almost_isotopic_to_diagonal}
There exists a compactly supported isotopy of Lagrangians in $\bC^{n}$ between the Polterovich surgery of $-(1,\ldots,1) +\i  \bR^{n}$ and $(\i , \ldots,\i ) +  \bR^{n}$,  and a Lagrangian $\Delta_{-1,\i }$ which agrees with the diagonal in a neighbourhood of $\bR^{n} \cup \i  \bR^{n}$.
\end{lem}
\begin{proof}
Translating co-ordinates, it suffices to prove that there is a Lagrangian submanifold, Lagrangian isotopic via a compactly supported isotopy, to a  Lagrange surgery of $\bR^n \cup\i  \bR^n$, and agreeing in some open neighbourhood of 
\[
-(1,\ldots,1) + \i \bR^n \, \cup \, (\i ,\ldots,\i ) + \bR^n
\] with the linear Lagrangian submanifold obtained by translating the diagonal in $\bR^n \oplus \i \bR^n$ by $(-1+\i , \ldots, -1+\i )$.

Recall the model of Equation \ref{Eqn:LagSurgery} for the surgery.  We choose the path $\gamma$, in the upper left quadrant, to satisfy the conditions:
\begin{enumerate}
\item $\gamma(0) = \sqrt{n}(-1+\i )$;
\item the tangent line to $\gamma$ at $\gamma(0)$ is $\bR\langle (1+\i )\rangle \subset \bC$;
\item the imaginary part im$(\gamma(t))$ is non-decreasing and monotonically increasing near $0$.
\end{enumerate}
Consider the associated submanifold $V_{\gamma} \subset \bC^n$.  It is straightforward to check that the first and third conditions imply
\[
V_{\gamma} \ \cap \ \left(  (-1,\ldots, -1)+\i \bR^n \cup (\i ,\ldots,\i ) + \bR^n \right)\ = \ \{ (-1+\i , \ldots, -1+\i ) \} 
\]
Hence this intersection is exactly $z\gamma(0)$, where $z = (1,\ldots, 1)/ \sqrt{n} \in S^{n-1}$. Moreover, this isolated intersection point  is transverse, and the tangent space to $V_{\gamma}$ at the intersection point agrees with the translated diagonal (by our prescription of the tangent space to $\gamma$ at $\gamma(0)$).  We now define $\Delta_{-1,\i}$ by linearising $V_{\gamma}$ in a small neighbourhood of this point.
\end{proof}

\begin{rem}
In dimension $n=2$, there is a Lagrange surgery of $(1+\i \bR) \cup (\i +\bR)$ which agrees with the \emph{antidiagonal} near the origin, and of $(-1+\i \bR) \cup (\i  +\bR)$ which agrees with the diagonal near the origin, cf. Figure \ref{fig:surgery_lagrangians}.  This accounts for the choice of translated axes above. (For any $n$, the space of linear Lagrangians transverse to both $\bR^n$ and $\i \bR^n$ has two components, distinguished by the Maslov index.)
\end{rem}

\begin{figure}
  \centering
  \includegraphics{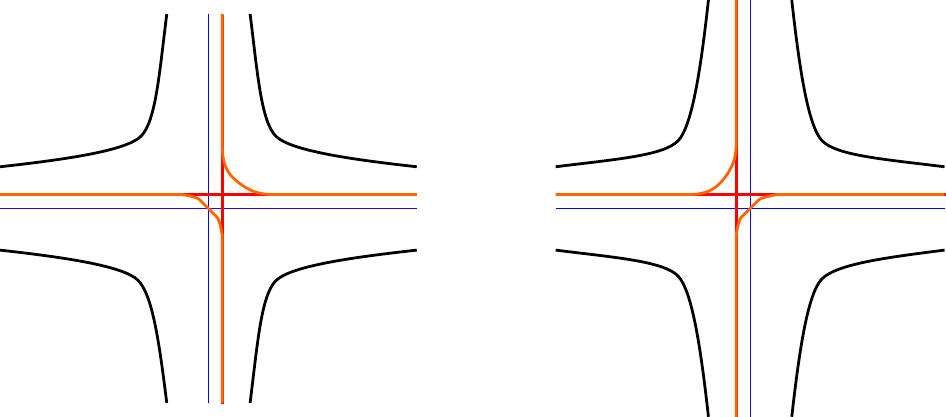}
  \caption{ }
  \label{fig:surgery_lagrangians}
\end{figure}

At this stage, we are ready to complete the proof of the main result of this section:
\begin{proof}[Proof of Theorem \ref{thm:wrapped_category_generation}]
Comparing the statement of the Theorem with that of Corollary \ref{cor:generation_with_diagonal}, we see that it suffices to prove that the diagonal gives rise to a Lagrangian in the plumbing which is quasi-isomorphic to an iterated cone on cotangent fibres of $Q_0$ and $Q_1$.  In fact, we shall show that it is isomorphic to a single cone $ L_{q} \to L_{p} $.

In order to prove this, we first choose $q$ so that the compactly supported isotopy in Lemma \ref{lem:surgery__almost_isotopic_to_diagonal} is contained in the model region, and abusively write $\Delta_{p,q}$ for the resulting Lagrangian in the plumbing.  Then we choose a smaller model region which intersects $\Delta_{p,q} $ in the diagonal, and attach this region to disc cotangent bundles of small radius to obtain a Liouville subdomain
\begin{equation}
   D^{*}_{\leq \epsilon}Q_0 \# D^{*}_{\leq \epsilon}Q_1  \subset    D^{*}Q_0 \# D^{*}Q_1.
\end{equation}
We now apply the restriction functor to the subcategory of $\Wrap(  D^{*}Q_0 \# D^{*}Q_1)$ with objects $L_{p}$, $L_{q}$, and $\Delta_{p,q}$.  Since $\Delta_{p,q}$ is quasi-isomorphic to the cone $L_{q} \to L_{p}$, the restriction of $\Delta_{p,q}$  is also quasi-isomorphic to such a cone.  We conclude that $\Delta^{\epsilon} = \Delta_{p,q} \cap    D^{*}_{\leq \epsilon}Q_0 \# D^{*}_{\leq \epsilon}Q_1 $ lies in the category generated by $L^{\epsilon}_{q} =  L_{q} \cap    D^{*}_{\leq \epsilon}Q_0 \# D^{*}_{\leq \epsilon}Q_1 $ and $L^{\epsilon}_{p} =  L_{p} \cap    D^{*}_{\leq \epsilon}Q_0 \# D^{*}_{\leq \epsilon}Q_1 $.  By construction, $ \Delta^{\epsilon}  $ is the diagonal Lagrangian, while $L^{\epsilon}_{q}   $ and $L^{\epsilon}_{p}   $ are respectively cotangent fibres to $Q_0$ and $Q_1$.  We complete the proof by noting that the wrapped Fukaya category of the plumbing is independent of $\epsilon$ because an appropriate Liouville flow rescales $  D^{*}Q_0 \# D^{*}Q_1$ into $D^{*}_{\leq \epsilon}Q_0 \# D^{*}_{\leq \epsilon}Q_1   $.
\end{proof}

\section{Generation by compact cores}
\label{sec:gener-comp-cores}

Let $M$ be the Liouville manifold obtained by plumbing  $M=T^*Q_0 \# T^*Q_1$.  From the previous section, we know that the Fukaya category of closed exact Lagrangian submanifolds of Maslov class zero embeds in the category generated by the cotangent fibres.  In this section, we shall prove

\begin{prop} \label{prop:generated-by-local-systems}
Let $L \subset M$ be a closed exact Lagrangian of Maslov class zero. Every finite rank local system over $L$  lies in the category generated by  finite rank local systems over $Q_0$ and $Q_1$.
\end{prop}

The proof of Proposition \ref{prop:generated-by-local-systems}, relies on two computations of wrapped Floer homology, respectively given in Sections \ref{sec:non-posit-masl} and \ref{sec:computation-degree-0}

\begin{prop} \label{prop:non-positive}
If $n \geq 3$,  there is a choice of gradings on the cotangent fibres $\L[i]$  such that the wrapped Floer cohomology algebra
\begin{equation}
  A^* \ = \ \oplus_{i,j} HW^{*}(\L[i] ,\L[j])
\end{equation}
is supported in  non-positive degree.  Moreover: 
\begin{enumerate}
\item if $n>3$, $A^0 = \oplus_i HW^0( \L[i], \L[i])$; 
\item if $n=3$, the subspace $A_{0,1}^{0} = HW^{0}(\L[0],  \L[1] )$ is an ideal on which all higher products vanish, whilst $A_{1,0}^0 = HW^{0}(\L[1],  \L[0] ) = 0$.  The quotient $A^* / A_{0,1}^0$ is the direct sum of $\oplus_i HW^0(\L[i], \L[i])$.
\end{enumerate}
\end{prop}
\begin{example}
Proposition \ref{prop:non-positive} is false when $dim(Q_i) = 1$.  Let $M$ be the plumbing of two copies of $T^*S^1$, so equivalently $M = T^2 \backslash \{pt\}$ is a punctured torus.  The Reeb flow on $\partial M \cong S^1$ is periodic, and there is a spectral sequence converging to $SH^*(M)$ for which the columns of the $E_1$-page are given by one copy of $H^*(M)$ and infinitely many copies of $H^*(\partial M)$, in suitably shifted degrees. From here Seidel computed \cite[Example 3.3]{Seidel:bias} that $SH^*(M)$ is non-trivial in infinitely many positive degrees. Analogously, one can directly compute the indices of Reeb chords to see that in this case $A^*$ is also non-trivial in infinitely many positive degrees.
\end{example}

For our application, we shall need to compute the degree $0$ part of the wrapped Floer cohomology of a cotangent fibre, and relate its category of (ordinary) modules to the category of representations of the fundamental group.    Recall that every local system on a Lagrangian gives rise to an object of the Fukaya category.   Since $Q_i$ intersects $\L[i] $  at one point,  a module $E_i$ over the group ring of $\pi_{1}(Q_i,q_i)  $  gives rise to a module over $ HW^{0}(\L[i] ,\L[i])  $ by considering $HW^*(\L[i], E_i)$; since $Q_i$ is disjoint from $ \L[j] $  if $i \neq j$, we obtain a representation of $A^*$ by projecting to $ HW^{0}(\L[i] ,\L[i]) $. 

\begin{prop} \label{prop:compute_HW_0_cotangents}
If $n\geq 3$, there is an isomorphism from  $HW^{0}( \L[i], \L[i] )$ to the group ring of $\pi_{1}(Q_i,q_i)$ such that the pullback of $E_i$ is isomorphic to $ HW^*(\L[i], E_i) $.  
\end{prop}
\begin{rem}
In the case where one of the Lagrangians $Q_i $ is a sphere, this Proposition can be derived from results of Bourgeois-Ekholm-Eliashberg \cite{BEE}. \end{rem}

Proposition \ref{prop:generated-by-local-systems} follows from Propositions \ref{prop:non-positive} and \ref{prop:compute_HW_0_cotangents} by application of the general theory of modules over non-positively graded $A_{\infty}$ algebras, which we review briefly.   The starting point is the Homological Perturbation Lemma of Kadeishvili (see \cite{kad}), which implies that any $A_{\infty}$ algebra $A$ is equivalent to one for which the differential vanishes.  Recall that an $A_{\infty}$ module is a  graded vector space $P$ together with a collection of maps 
\begin{equation}
\mu^{1|d} \co  P \otimes A^{d}\to P  
\end{equation}
of degree $1-d$ satisfying a homotopy associativity equation; Homological Perturbation similarly implies any such is equivalent to one with $\mu^{1|d} \neq 0$ only for $d>0$.  We shall call such algebras and modules \emph{minimal}.

The main structure result for a minimal $A_{\infty}$ algebra $A$ supported in non-positive degrees is the following, cf. related ideas in \cite{DGI}:
\begin{lem} \label{lem:filtration_degree}
The ascending degree filtration on any module is a filtration by $A_{\infty}$-submodules.  The subquotients, which are supported in a single degree, are  determined up to $A_{\infty}$ equivalence by the corresponding representation of $A^{0}$.
\end{lem}

\begin{proof}
Under the assumption of minimality, all the $A_{\infty}$-operations $\mu^k$ have degree $2-k \leq 0$, and only the product has degree $0$. Similarly, all the $\mu^{1|d}$ have degree $<0$ except for the multiplication $\mu^{1|1}$. The rest is straightforward.
\end{proof}

\begin{proof}[Proof of Proposition \ref{prop:generated-by-local-systems}]
By Theorem \ref{thm:wrapped_category_generation}, it suffices to prove that every module over $A^*$ which is supported in finitely many cohomological degrees has a filtration by such modules induced by local systems on $Q_0$ and $Q_1$.  It follows from the discussion above that every such module has a filtration by modules supported in a single cohomological degree, and that such a module is completely determined by the representation of $A^0 = \oplus_{i,j} HW^{0}( \L[i], \L[j]) $.  If $n>3$, this group is a isomorphic to the sum of group rings $\oplus_i \pi_{1}(Q_i,q_i)$, and the module action is the obvious one, by Proposition \ref{prop:compute_HW_0_cotangents}.  Hence, every such module arises from a local system on $Q_i$.

When $n=3$ there is an additional step. Recall that $A_{0,1}^{0} = HW^{0}(\L[0],  \L[1])$ is an ideal in $A^{0}$.  Given any module $P$ over $A^{0}$, $A_{0,1}^{0} \cdot P$ is a submodule on which $A_{0,1}^{0}$ acts trivially because the product vanishes on it.  Since $A_{0,1}^{0}$ also acts trivially on $P / \left( A_{0,1}^{0} \cdot P   \right)$, we may write $P$ as an extension
\begin{equation}
  P / \left(A_{0,1}^{0} \cdot P\right) \to A_{0,1}^{0} \cdot P
\end{equation}
of two modules on which $A_{0,1}^{0}$ acts trivially.  Each of these is determined by its structure as a module over $ \oplus_{i} HW^{*}( \L[i] ,\L[i]) $.  By the same argument as above, such a module corresponds to a local system on $Q_i$.
\end{proof}

We continue to assume that $A^*$ is graded in non-positive degrees.
The following gives some control on the morphisms in a twisted complex representing a suitable shift of a closed exact Lagrangian, and was used in the discussion after the statement of Proposition \ref{prop:braid-orbit}.

\begin{lem}\label{lem:degree-at-least-2}
Any finite rank module over $A^*$ is equivalent to a twisted complex  $\cC^{\bullet} =(\oplus_{i} \bZ[i] \otimes P^{-i}, \delta_{ij})$, with each $P^{-i}$ an ordinary $A^0$-module and $\bZ[i]$ a free graded abelian group concentrated in degree $-i$.   The morphisms $\delta_{ij}$ are given by elements of $Hom(P^{-i}, P^{-j})$ of degree $\geq 2$.
\end{lem}
\begin{proof}
Applying the homological perturbation lemma, we assume that such a module $P$ is minimal, and define $P^{k}$ to be its component in degree $k$. The ascending degree filtration
\begin{equation}
  P^{(l)} = \bigoplus_{k \leq l} P^{k}
\end{equation}
 is a filtration by submodules. Note that we have an exact triangle
\begin{equation}
  \xymatrix{P^{(l-1)} \ar[rr] & &  P^{(l)}  \ar[ld] \\
& \bZ[-l] \otimes P^{l} \ar[lu] .&  }
\end{equation}
In particular, $ P^{(l)}  $  is quasi-isomorphic to the cone of a morphism
\begin{equation}
  \bZ[-l] \otimes P^{l} \to P^{(l-1)}.
\end{equation}
Since $P$ is of finite rank, the degree filtration terminates, and we obtain, by induction, a description of $P$ as a twisted complex.   Note that, in our convention for a twisted complex, the differential $\delta_{i,j}$ is non-vanishing only if $i<j$.  We therefore define the $i$\th part of this complex to be $   \bZ[i] \otimes P^{-i} $.   Since the morphisms $\delta_{ij}$ have total degree $1$, and the homological shift is strictly negative, the corresponding morphism in $ Hom(P^{-i}, P^{-j}) $ has degree strictly greater than $1$.
\end{proof}

\begin{rem} \label{rem:generation_infinite_local_systems}
Lemma \ref{lem:filtration_degree} was implicitly concerned with finite rank modules, in particular finite rank local systems over exact Lagrangians, but there is a version which holds for arbitary rank systems. The relevant result is essentially \cite[Proposition 3.3]{Abouzaid-ExactLag}.  We fix a Liouville manifold $M$, a collection of (not necessarily closed) exact  Lagrangian submanifolds $\{L_i\}$,  and a further collection of closed exact Lagrangians $\{Q_j\}$, satisfying:
  \begin{enumerate}
  \item the $L_i$ split-generate the Fukaya category;
  \item the algebra $A^* =\oplus_{i,j} HW^*(L_i, L_j)$ is supported in non-positive degrees;
  \item every indecomposable $A^0$-module is isomorphic to that defined by some (not necessarily finite  rank) local system $V_j \rightarrow Q_j$ over one of the $Q_j$.
  \end{enumerate}
 Then \cite[Proposition 3.3]{Abouzaid-ExactLag} asserts that every (arbitrary rank) local system over a closed exact Lagrangian submanifold of $M$ can be written as a finite twisted complex over (arbitrary rank) local systems on the $Q_j$. It is this stronger result which is needed to prove that exact Lagrangians in the $A_2^n$-spaces are actually simply connected.
\end{rem}

\subsection{Neck stretching} \label{sec:neck-stretching}

The final three sections are devoted to the proofs of Proposition \ref{prop:non-positive} and Proposition \ref{prop:compute_HW_0_cotangents}. 
For a single manifold $Q$, the fact that $HW^*(T_q^*Q, T_q^*Q)$ is supported in non-positive degree arises from the relation between the Maslov index of binormal chords in a cotangent bundle, and the Morse index of the corresponding critical points of the path space action functional.  This relation reduces to the fact that a curve in the Lagrangian Grassmannian of $T^*Q$ defined by parallel transport of the vertical distribution along a curve in $T^*Q$ of the form $(q(t), \dot{q}(t))$ meets the Maslov cycle always with the same co-orientation \cite{Mane}.  To obtain an analogue for plumbings, we begin by equipping $Q_i$ with metrics $g_i$ which can be written as
\begin{equation}
f(|x|) g_{eucl}
\end{equation}
with $f$ a smooth function which is identically $1$ near the origin, in some charts $U_{i}$ centered at $q_i$, whose domains are balls of radius $e^{2}$ in $\bR^{n}$.  Moreover, we require that
\begin{equation} \label{eqn:charts-for-plumbing}
  \parbox{30em}{$f(|x|) = \frac{1}{|x|^2}$ whenever $1<|x|$.}
\end{equation}
This condition on $f$ shall be used to facilitate ``stretching the neck:''
\begin{lem}
  The complement of the origin in $\bR^{n}$, equipped with the metric
  \begin{equation}
     \label{eq:rescaled_metric}
 \frac{g_{eucl}}{ |x|^2}. 
  \end{equation}
is isometric to the product metric on $\bR \times S^{n-1}$.
\end{lem}
\begin{proof}
Using radial coordinates, we may write every point in the complement of the origin as $r \theta$, with $\theta \in S^{n-1}$ and $ r \in (0,+\infty)$.  The map
\begin{equation}
  (s,\theta) \mapsto e^{s} \theta
\end{equation}
provides the desired isometry between the product metric and Equation \eqref{eq:rescaled_metric}.
\end{proof}

\begin{figure}[ht]
  \centering
   \includegraphics{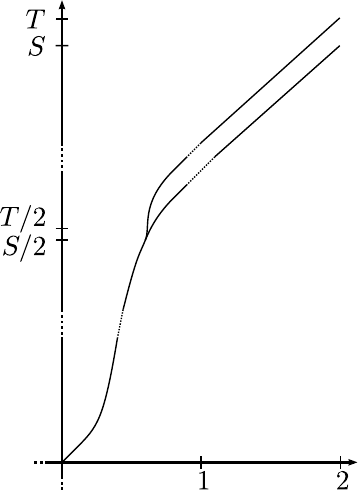}
  \caption{Graphs of the functions $\chi_{T}$ and $\chi_{S}$}
  \label{fig:chi_function}
\end{figure}

With this in mind, we find that the Riemannian manifold $(Q_i,g_{i})$ admits an isometric embedding of $[0, 2] \times S^{n-1}$.  We define a family of metrics $g_{i}^{T}$  which agrees with $g_i$ away from this region, and which, in this region is stretched, so that it admits an isometric embedding of $[0, T+2] \times S^{n-1}$.  More precisely, let us fix the coordinates 
\[
(2-\log(|x|), \theta) \quad \textrm{on the region} \ \{1 < |x| < e^{2}\} \subset U_i.
\] 
Note that in these co-ordinates, $\{0\} \times S^{n-1} \subset [0,2] \times S^{n-1}$ corresponds to the sphere of radius $e^2$ around $q_i$, whilst $\{2\}\times S^{n-1}$ corresponds to the sphere of radius 1.  In particular, a neighbourhood of $0$ in the collar $[0,2] \times S^{n-1}$ extends the \emph{complement} of $U_i$, at the \emph{outer} boundary of the chart  centred on $q_i$. 
In these coordinates, the metrics $g_{i}^{T}$ are defined as
\begin{equation}
\frac{d \chi_{T}}{ds}  g_{[0,2]} \oplus g_{S^{n-1}}
\end{equation}
where $\chi_{T} \co [0,2] \to \bR$ is a family of strictly monotone functions satisfying the following conditions (see Figure \ref{fig:chi_function}):
\begin{enumerate}
\item $\chi_{T} \equiv  s $ in a neighbourhood of $0$.  This condition ensures that $g_i$ extends this metric smoothly to the complement of $U_i$.
\item $\chi_{T} \equiv T+s $ on a neighbourhood of $[1,2]$.  This condition implies that a path in the radial direction between the two boundary spheres has length $T+2$, and that $g_i$ extends this metric smoothly to the rest of $U_i$.
\item $\chi_{S}^{-1}(0,S/2) = \chi_{T}^{-1}(0,S/2)$ whenever $S < T$.  Let $ Q_i^{T}$ denote the union of $ \chi_{T}^{-1}(0,T/2)$ with the complement of $U_i$.   From this condition, we conclude that we have an inclusion $  Q_i^{S} \subset  Q_i^{T} $ whenever $S < T$, and that the restriction of $g_{i}^{T}$ and $g_{i}^{S}$ to $ Q_i^{S}  $  agree.
\item If $s < 1 $, $\lim_{T \to +\infty}\chi_{T}(s)$ is well defined and finite.  Let $Q_{i}^{op}$ denote the complement of the set of points in $U_i$ which are the images of points with norm $ |x|  \leq e $.  This final condition implies
  \begin{equation} \label{eq:increasing_union}
    \cup_{T} Q_i^{T} = Q_{i}^{op}.
  \end{equation}
\end{enumerate}

We now consider the family of Liouville domains associated to these metrics.  First, we set the ``plumbing region"
\begin{equation} \label{eq:model_region}
  R^0 = \{ |x|^2 |y|^2 \leq 1 , |x| < e, |y|< e\} \subset \bR^n \times \i \bR^n,
\end{equation}
and we define 
\begin{equation}
  M = D^*Q_0^{op} \cup R^0 \cup D^*Q_1^{op}
\end{equation}
where the unit cotangent bundle is taken with respect to the metric $g_i$.  The reader may check that this definition makes sense, by computing that the unit disc cotangent bundle of the product $[0,1] \times S^{n-1}$, isometrically embedded in $Q_i $ with respect to the metric $g_i$, is symplectomorphic to the set of points in $R^0$ for which the norm of the $x$ variable lies in the interval $[1, e]$.   

Since $\partial M$ is a smooth contact manifold, we have an associated complete Liouville manifold $\hat{M} = M \cup \partial M \times [1,+\infty)$.   Note that, by construction, we have (disjoint) embeddings $T^{*}  Q_i^{op} \subset  \hat{M} $.  With this in mind, we may construct a Liouville subdomain
\begin{equation}
  M^{T} \subset \hat{M}
\end{equation}
by taking the union
\begin{equation} 
  M^{T} = D^*_{T}Q_0^{op} \cup R^0 \cup D^*_{T}Q_1^{op}
\end{equation}
where $ D^*_{T}Q_i^{op}  $ is the unit cotangent bundle with respect to the metric $g_{i}^{T}$.  Note that $\hat{M}$ is also symplectomorphic to the completion of $M^{T}$.   In particular, all Floer homological invariants are independent of $T$.

There is a slightly different description of the manifolds $M^{T}$ which is related to the idea of \emph{stretching the neck}.  For $T \in \bR_{>0}$, we may define a plumbing region 
\begin{equation} \label{eq:model_region}
  R^T = \{ |x|^2 |y|^2 \leq 1 , |x| < e^{T+1} , |y|< e^{T+1}\} \subset \bR^n \times \i \bR^n.
\end{equation}
The set of points whose $x$ coordinate lies in the interval $[1, e^{T+1}]$ can be identified with the unit disc cotangent bundle of the product $[0,T+1] \times S^{n-1}$, which, by construction, isometrically embeds in $Q_i$ with respect to the metric $g_i^{T}$.  In particular, $M^{T}$ has an alternative description as
\begin{equation} \label{eq:stretched_plumbing}
 D^*Q_0^{T} \cup R^{T} \cup D^*Q_1^{T}.
\end{equation}

\begin{rem} \label{rem:moon}
The presentation $R^{\infty} = \{ |x|^2 |y|^2 \leq 1\}$ for the plumbing region suggests a misleading symmetry exchanging $x$ and $y$.  However, the symplectic form $dx \wedge dy$ is negated under this symmetry, and the corresponding Hamiltonian flow is therefore reversed to give negative geodesic flow if one views $Q_0$ as the local zero-section. For instance, consider two cotangent fibres $\L[i]$. We claim there is a Hamiltonian chord contained in $\partial R^{\infty}$ between the boundary Legendrians of these fibres in at most \emph{one} of the two possible directions. The result is particularly vivid in the flat case, $n=2$, illustrated in Figure \ref{fig:dark_side_moon}; from the viewpoint of the $Q_0$-plane, one Legendrian boundary projects to a point $q_0$, whilst the other projects to a sphere $\{|q| = 1/|q_1|\}$, and the possible chords are constrained by the associated tangent data.
\end{rem}
 \begin{figure}
  \centering
  \includegraphics{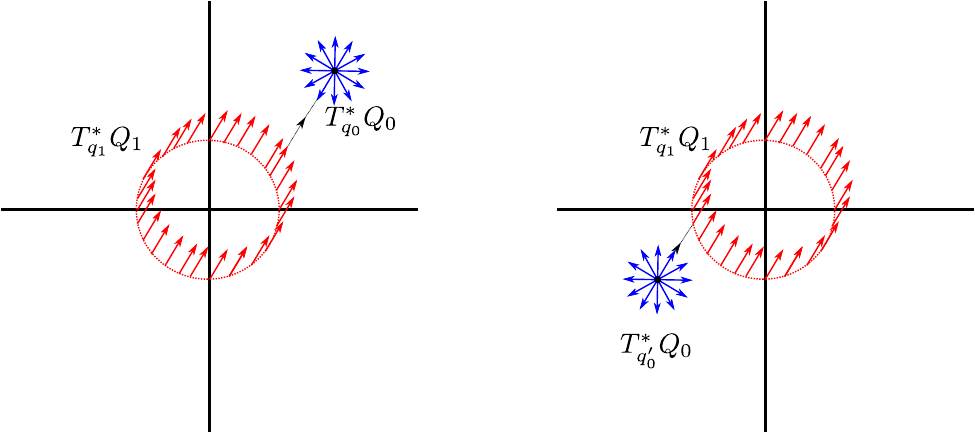}
  \caption{Indeed, the moon has a dark side.}
  \label{fig:dark_side_moon}
\end{figure}

\subsection{Non-positivity of Maslov indices in plumbings} \label{sec:non-posit-masl}

In the plumbing, the Maslov index measures twisting along a Hamiltonian chord of a Lagrangian foliation which interpolates between the usual vertical foliations in the two pieces,  see Figure \ref{fig:foliation_dim_1} for a picture in the lowest-dimensional case.  Since it is hard to make these Lagrangian foliations completely explicit in dimension $>2$, we will reduce the computation of indices to a problem involving holomorphic disks.

By applying a suitable conformal dilation to the symplectic form, we now work with the Liouville domain $M$ which contains a standard plumbing region $R^{\infty} \cap B_0(2) \subset \bC^n$.
Fix a Reeb chord $\gamma$  from $\L[0]$ to $\L[1]$ inside the boundary of $M$, where by standard convention these cotangent fibres are disjoint from $R^{\infty} \subset M$.  We will decompose $\gamma$ into a sequence of chords $\gamma = \gamma_1 \ast \cdots \ast \gamma_n$ where each $\gamma_i$ is
\begin{itemize}
\item either a segment of a geodesic inside the (vertical) boundary of the open region $D^*Q_i^{op}$ for $i=0,1$;
\item or a Reeb chord for the Hamiltonian flow on $\partial R^{\infty}$ for the model region described above.
\end{itemize}

In particular, once $\gamma$ is given, one can choose a level set $|x|=1+\varepsilon$ of the radial function in the local model which $\gamma$ crosses transversely.  In particular, we may find $0< \delta < \varepsilon$ such that $\gamma$ crosses every level set between $1+\varepsilon  $ and $ 1+\delta $ transversely.  Under this assumption, every time the geodesic crosses the level set $|x|=1+\varepsilon$, it reaches $|x|=1+\delta$, and vice versa depending on whether $\gamma$ is leaving or entering $D^{*} Q_{0}$.  Following \cite{plumbings}, we fix a Lagrangian foliation $\sL \subset TM$, depending on $\gamma$, with the following properties:

\begin{itemize}
\item  $\sL$ co-incides with the vertical foliations near $S^*Q_i$ except in the collar region $|x| \in (1+\delta, 1+\varepsilon)$; 
\item $\sL$ interpolates smoothly between the vertical foliations in the collar.
\end{itemize}  

$\sL$ will be tangent to the boundary at some points inside the collar, cf. Figure \ref{fig:foliation_dim_1}.
\begin{figure}
  \centering
  \includegraphics{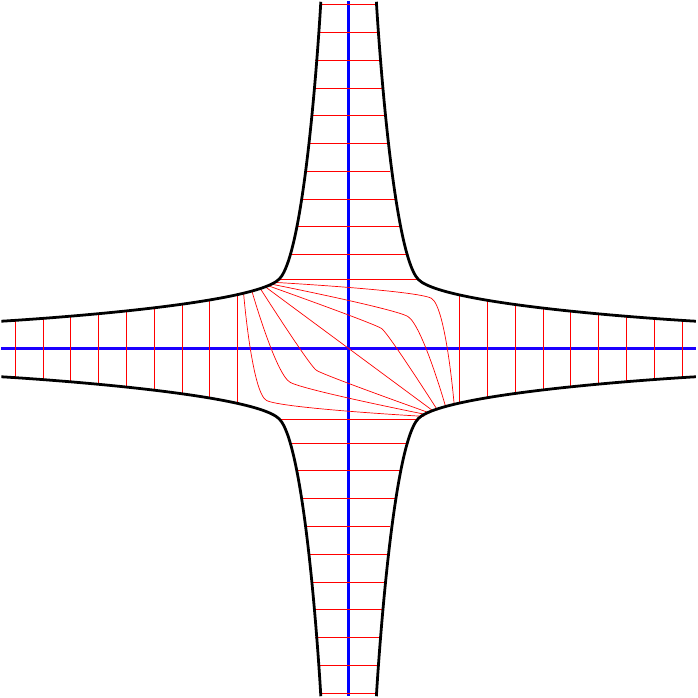}
  \caption{ }
  \label{fig:foliation_dim_1}
\end{figure}
All the foliations (for different chords or slicings) can be taken to be homotopic, hence all define the same Maslov indices.  For any choice of gradings of the cotangent fibres $\L[i]$, the foliation $\sL$ defines a canonical Maslov index $\mu(\gamma) \in \bZ$.  Moreover, each $\gamma_i$ is a Reeb chord between Legendrian boundaries of disk cotangent fibres, hence again has a uniquely defined Maslov index.  The first claim is that the index of the whole chord is governed by the indices of the pieces (this is presumably well-known, compare to \cite[Theorem 2.3]{Robbin-Salamon} for a corresponding local statement).

\begin{lem} \label{lem:Maslov_index_additive}
$\mu(\gamma) = \sum_j \mu(\gamma_j)$.
\end{lem}

\begin{proof}
Fix a Hamiltonian chord $\hat{\gamma}: [0,1]\rightarrow \partial M$ between Legendrian submanifolds $\Lambda_i \ni \gamma(i)$.  There is a unique trivialisation of $\hat{\gamma}^*TM$ which takes the Lagrangian foliation $\sL$ to the constant real subbundle $\hat{\gamma} \times \bR^n \subset \hat{\gamma} \times \bC^n$.   The Maslov index of this Reeb chord is given by the index of a $\overline{\partial}$-operator on the half-plane (disk with one boundary puncture), with Lagrangian boundary condition given by the image under parallel transport by the Hamiltonian flow of the tangent plane to $\Lambda_0$. 

Now consider a pair of Reeb chords $\hat{\gamma}_+$ and $\hat{\gamma}_-$ with the property that $\hat{\gamma}_+(1) = \hat{\gamma}_-(0)$ and with all 3 endpoints of the chords lying on Legendrian submanifolds whose tangent spaces are contained in the Lagrangian foliation $\sL$. Let $\hat{\gamma}_{\circ}$ denote the obvious concatenation of $\hat{\gamma}_{\pm}$.  The usual additivity of the index for $\overline{\partial}$-operators implies that 
\[
\mu(\hat{\gamma}_{\circ}) - (\mu(\hat{\gamma}_+) + \mu(\hat{\gamma}_-)) \ = \ index(\mathcal{P})
\]
where $\mathcal{P}$ is the operator given by counting holomorphic sections over a three-punctured disk with boundary conditions given by the end-points of the three chords appearing on the left side of the equation. Since all these boundary conditions lie in the foliation $\sL$, in the given trivialisation the Lagrangian boundary condition on this pair of pants is constant, hence it has index 0.
\end{proof}

To prove Proposition \ref{prop:non-positive}, it suffices to show that the indices of the $\gamma_j$ are all non-positive when $dim(Q_i) \geq 2$.  
By our choice of foliation and slicing of $\gamma$, for those Reeb chords contained in the $D^*Q_i^{op}$ this follows from the usual non-positivity (in our convention) of indices for geodesic chords in any Riemannian metric.

 The remaining pieces $\gamma_i$ of chord are those which traverse the plumbing region; in other words, which crosses the region separated by the hypersurfaces $|x| = 1 + \varepsilon$ and $ |x| = 1 + \delta$.   Choosing $\epsilon$ and $\delta$ sufficiently close, these chords correspond to short geodesics on the product of an interval with the sphere, which start and end on different boundary components.

A Hamiltonian chord between $L_0$ and $L_1$ for the time-1 flow $\phi_H^1$ is also an intersection point of $L_0$ and $\phi_H^1(L_1)$, so the problem of determining the index of a chord is equivalent to the problem of determining the index of a Lagrangian intersection. In the wrapped category, indices for such intersection points are invariant under arbitrary (not necessarily compactly supported) Lagrangian isotopies which do not change the configuration of intersection points of the two Lagrangians.  
 In particular, we can deform the cotangent fibres near the boundary of the plumbing inwards towards the origin until they intersect, and by continuity the index of the unique Reeb chord will then be given by the index of the unique Lagrangian intersection point of the deformed fibres.  Take these for definiteness to be the fibres through $(1/2,0,\ldots,0)$ and $(\i/2, 0, \ldots, 0)$ in the local model $R^{\infty}$.
 
\begin{lem} \label{lem:quadrilateral}
There is a rigid holomorphic quadrilateral with boundary edges on the linear Lagrangian subspaces
\[
\bR^n, \ \i \bR^n, \ \bR^n + (\i /2,0,\ldots,0), \ \i \bR^n + (1/2,0,\ldots,0).
\]
\end{lem}

\begin{proof}
Projecting to the co-ordinate axes $\bC \subset \bC^n$, the boundary conditions define the four distinct lines $\bR, \i \bR, \bR+\i /2,\i \bR+1/2 \subset \bC$ in the first factor, and the real and imaginary axes (each repeated) in the remaining $n-1$ factors.  Any holomorphic quadrilateral with these boundary conditions is therefore constant under projection to the co-ordinates $z_2,\ldots, z_n$, whilst there is a unique rigid quadrilateral under projection to the $z_1$-plane.  This fixes the modulus of the quadrilateral, and the corners are all convex, so the curve is isolated and transversely cut out.
\end{proof}

We now grade the compact cores and cotangent fibres of the plumbing as follows:

\[
HF(Q_i, \L[i]) \cong k[0]; \quad HF(Q_0, Q_1) \cong \begin{cases} k[n/2] \quad n \ \textrm{even} \\ k[(n-1)/2] \quad n \ \textrm{odd} \end{cases}
\]
The other gradings are determined by Poincar\'e duality, so
\[
HF(\L[i], Q_i) \cong k[n]; \quad HF(Q_1, Q_0) \cong \begin{cases} k[n/2] \quad n \ \textrm{even} \\ k[(n+1)/2] \quad n \ \textrm{odd} \end{cases}
\]

Via construction, cf. Definition \ref{defin:plumbing}, open subsets of these cores and fibres are identified with open disks in the linear Lagrangian subspaces of Lemma \ref{lem:quadrilateral}.
Rigidity of the quadrilateral described in Lemma \ref{lem:quadrilateral} means that it contributes to a non-trivial higher multiplication
\[
\mu^3: HF(\L[1], Q_1) \otimes HF(\L[0], \L[1]) \otimes  HF(Q_0, \L[0]) \ \longrightarrow \ HF(Q_0, Q_1)[-1]
\]
which in turn forces the group $HF(\L[0], \L[1])$ to be concentrated in degree 
\begin{equation} \label{eqn:final-degrees}
HF(\L[0], \L[1]) \cong \begin{cases} k[(2-n)/2] \quad n \ \textrm{even} \\ k[(1-n)/2] \quad n \ \textrm{odd} \end{cases}
\end{equation}
If $n>3$, it follows that for this (maximally symmetric) choice of grading, the chord from $\L[0]$ to $\L[1]$, or the shortest chord in the reverse direction as appropriate, cf. Remark \ref{rem:moon}, has strictly negative degree.  If $n=3$, one chord has negative degree whilst the other has degree $0$, whilst if $n=2$ the chords both have degree $0$.

At this stage, we conclude that the algebra $A^* = \oplus_{i,j} HW^*(\L[i], \L[j])$ is graded in degrees $\leq 0$ whenever $n\geq 2$, and its degree $0$ part is a quotient of $\oplus_i HW^0(\L[i], \L[i])$ whenever $n\geq 4$.  There is a final delicacy when $n=3$, which is that there may still be additional chords of degree $0$ not coming from loops in any component; however, (\ref{eqn:final-degrees}) implies these chords can arise only in one direction, i.e. from $Q_1$ to $Q_0$ but not vice-versa.  Therefore $A^0$ differs from a quotient of $\oplus_i HW^0(\L[i], \L[i])$ by a nilpotent ideal, as required.  (An alternative would be to restore symmetry to the grading of $A^*$ in \eqref{eqn:final-degrees} by taking gradings in $(1/2)\bZ$, in which case the argument when $n=3$ becomes formally identical to the higher-dimensional case.)

\subsection{Computation of the degree $0$ part of wrapped Floer cohomology}
\label{sec:computation-degree-0}
We shall prove Proposition \ref{prop:compute_HW_0_cotangents} in this Section.  The main point is to adapt to our setting the results of Abbondandolo and Schwarz \cite{ASchwarz} and of \cite{wrapped-based}, proving an isomorphism between wrapped Floer cohomology of cotangent fibres and the homology of the based loop space.

\subsubsection{From Floer theory to the homology of the based loop space}
Let $M$ be a Liouville manifold, and $Q \subset M$ a closed exact Lagrangian, which is an object of its Fukaya category.  The main result of \cite{wrapped-based} (see the discussion preceding Corollary 1.5) is that there exists a functor from the Fukaya category of $M$ to the category of modules over the chains on the based loop space of $Q$.  Let us consider the special case when a Lagrangian $L$ meets $Q$ at one point $q$.   Given a time-$1$ Hamiltonian chord $x$ with endpoints on $L$, define
\begin{equation}
  \cH(x)
\end{equation}
to be the space of pseudo-holomorphic maps
\begin{equation} \label{eq:positive_half_cylinder}
  u \co [0,1] \times [0,+\infty) \to M
\end{equation}
which asymptotically converge to $x$ as the second coordinate goes to $+\infty$, and with boundary conditions
\begin{equation}
  \parbox{30em}{$u(0,s) \in L$, $u(t,0) \in Q$, and $u(1,s) \in L$.}
\end{equation}
Note that $u(0,0) = u(1,0) = q$, since this is the only point lying both on $Q$ and $L$.  In particular, the path $t \mapsto u(t,0) $ is a based loop on $Q$.  In this special case, Lemma 4.5 \cite{wrapped-based}, asserts that the evaluation map
 \begin{equation}
    \ev \co  \cH(x)  \to \Omega_{q}(Q)
 \end{equation}
defines, at the level of complexes, a chain map
\begin{equation} \label{eq:map_CW_C_based_loops}
  CW^{*}( L) \rightarrow C_{-*}( \Omega_{q}Q). 
\end{equation}
Now, let $E$ be a local system on $Q$, with fibre $E_{q}$ at $q$.  Every local system defines an object of the Fukaya category (see \cite{Abouzaid-ExactLag} for a detailed description), and, in this situation, the fact that $L$ and $Q$ intersect at one point implies that we have a natural isomorphism
\begin{equation} \label{eq:iso_HW_fibre}
    HW^{*}(L,E) \cong E_{q}.
\end{equation}
The product in the Fukaya category
\begin{equation}
  HW^{*}(L,E) \otimes HW^{*}(L) \to HW^{*}(L,E) 
\end{equation}
therefore makes $E_{q}$ into a module over $ HW^{*}(L)  $; for degree reasons only $ HW^{0}(L) $ acts non-trivially.  This product counts maps \eqref{eq:positive_half_cylinder} which are rigid.  Therefore, given a generator $x \in CW^{0}(L)$, the associated map
\begin{equation}
  E_{q} \to E_{q}
\end{equation}
is obtained by taking the sum of the monodromy of the local system $E$ over all based loops on $Q$ associated to elements of $ \cH(x)  $.    Note that these monodromy maps make $E_{q}$ into a module over $H_{-*}( \Omega_{q}Q)$; again, for degree reasons, only $H_{0}( \Omega_{q}Q)$, which is the group ring of $\pi_1(Q,q)$, acts non-trivially.  We conclude that the $HW^{0}(L)$-module structure on $E_{q}$ is obtained by pulling back the $\bZ[\pi_1(Q,q)] $-module structure by the map induced by Equation \eqref{eq:map_CW_C_based_loops}  on homology. More precisely:
\begin{lem}
 The product in Floer theory, the isomorphism in Equation \eqref{eq:iso_HW_fibre}, and the map in Equation \eqref{eq:map_CW_C_based_loops} fit into a commutative diagram:
  \begin{equation}
    \xymatrix{ HW^{0}(L)  \otimes   HW^{0}(L,E) \ar[r] \ar[d] &  HW^{0}(L,E) \ar[d]^{\cong} \\
\bZ[\pi_1(Q,q)] \otimes  E_{q} \ar[r] &  E_{q}.}
  \end{equation}
  \qed
\end{lem}

Returning to the manifolds of interest, note that $\L[i]$ and $Q_{i}$ do intersect at one point.  In order to prove Proposition \ref{prop:compute_HW_0_cotangents}, it remains to show that the map induced by Equation \eqref{eq:map_CW_C_based_loops} yields an isomorphism
\begin{equation} \label{eq:map_to_group_ring}
  HW^{0}(\L[i]) \cong  H_{0}( \Omega_{q_i}Q_i) \cong \bZ[\pi_1(Q_i,q_i)].
\end{equation}
We shall prove this by exhibiting an inverse.

\subsubsection{Action functionals on based paths}
Let $\Omega_{q_i}(Q_i)$ denote the space of loops on $Q_i$, based at a point $q_i$, of class $W^{1,2}$.  Abbondandolo and Schwarz have related the Morse homology of the action functional on the spaces of such based loops to the wrapped Floer homology of cotangent fibres.   In the discussion below, we shall essentially follow \cite[Section 2]{ASchwarz}, modifying certain points to account for the fact that we need to compute the Floer homology in the plumbing.

We define the action of a path to be the integral
\begin{equation}
  \Energy^{T}(\gamma) = \int_{0}^{1} g^{T}\left( \frac{d \gamma}{dt} ,\frac{d \gamma}{dt}  \right) dt.
\end{equation}

\begin{lem} \label{lem:short_loops_independent_of_T}
If $S < T$, the set of based loops of action equal to $S^{2}$ is independent of $T$.
\end{lem}
\begin{proof}
The Cauchy-Schwartz inequality implies that
\begin{equation}
 \textrm{Length}(\gamma)^{2} \leq   \Energy^{T}(\gamma),
\end{equation}
hence a loop of action bounded by $S^2$ has length bounded by $S$.  Since the basepoint lies in the complement of the cylindrical region, every path which has length less than $S$ cannot escape a cylinder of length $S/2$.  From our choice of metric in Section \ref{sec:neck-stretching}, this implies that such a loop must be contained in $Q_i^{S}$, where the metrics  $g_{i}^{T}$ are independent of $T$ whenever $S < T$.
\end{proof}

Let us write $\Omega_{q_i}^{S}(Q_i; g_{i}^{T}) $ for the set of based loops of action bounded by $S^2$ with respect to the metric $g_{i}^{T}$.  By the above Lemma, we have an inclusion
\begin{equation}
  \Omega_{q_i}^{S}(Q_i; g_{i}^{S}) \subset \Omega_{q_i}^{T}(Q_i; g_{i}^{T})
\end{equation}
whenever $S < T$.  In particular, we can take the increasing union
\begin{equation}
  \cup_{T}  \Omega_{q_i}^{T}(Q_i; g_{i}^{T}) \subset  \Omega_{q_i}(Q_i).
\end{equation}
Note that all the loops in the above increasing union lie in $Q_{i}^{op}$.  Moreover, every loop in $Q_{i}^{op}$ has uniformly bounded length with respect to the metrics $g_{i}^{T}$ by Equation \eqref{eq:increasing_union}.  We conclude that
\begin{equation} \label{eq:increasing_union_loops_complement}
  \cup_{T}  \Omega_{q_i}^{T}(Q_i; g_{i}^{T}) = \Omega_{q_i}(Q_i^{op}) \subset \Omega_{q_i}(Q_i).
\end{equation}

\subsubsection{Morse homology of based paths}
The critical points of $\Energy^{T}$ are based geodesics, and the gradient flow of $\Energy$ is known to satisfy the Palais-Smale condition (see \cite[Proposition 2.5]{ASchwarz}) and the Morse homology of $\Energy^{T}$ computes the homology of the space of loops based at $q_i$.   More precisely, the homology of $  \Omega_{q_i}^{S}(Q_i; g_{i}^{T}) $  is computed by the homology of the subcomplex
\begin{equation}
  CM_{*}^{S}( \Energy^{T} ) \subset CM_{*}( \Energy^{T} )
 \end{equation}
generated by geodesics whose action is bounded by $S^2$. 

Lemma \ref{lem:short_loops_independent_of_T}, together with the fact that the negative gradient flow of $\Energy^{T}$ preserves $ \Omega_{q_i}^{S}(Q_i; g_{i}^{T})   $, imply that we have a canonical identification
\begin{equation}
  CM_{*}^{S}( \Energy^{T} ) = CM_{*}^{S}( \Energy^{T'} )
\end{equation}
whenever $T$ and $T'$ are both greater than $S$.   By Equation \eqref{eq:increasing_union_loops_complement}, the direct limit of the homology groups $HM^{T}_{*}( \Energy^{T} )  $ is the homology of the based loops in $Q_{i} - U_{i}$, with $U_i$ the chart of \eqref{eqn:charts-for-plumbing}:
\begin{equation}
 \lim_{T} HM_{*}^{T}( \Energy^{T} ) = H_{*}\left( \Omega_{q_i}(Q_i - U_{i}) \right).
\end{equation}

\subsubsection{Surjection in degree $0$}
Let $Q$ be a compact smooth manifold, $H$ a Hamiltonian on $\TQ$, and $J$ an almost complex structure on $\TQ$.  Given a chain $\sigma $ of loops on $Q$ based at $q \in Q$, and a time-$1$ Hamiltonian chord $x$ of $H$ with both endpoints on $T^{*}_{q}Q$, define
\begin{equation}
  \Moduli(x; \sigma)
\end{equation}
to be the moduli space of pseudo-holomorphic half-cylinders
\begin{equation} \label{eq:negative_half_cylinder}
  u \co [0,1] \times (-\infty,0] \to \TQ
\end{equation}
such that the projection of $u(t,0)$ to $Q$ is an element of $\sigma$, and $u$ converges, at $-\infty$ to $x$.  If $y$ is a generator of $ CM^{*}(\Energy) $, and $W^{u}(y )$ is the descending manifold with respect to the gradient flow of $\Energy$, Abbondandolo and Schwarz prove that the count of rigid elements of $ \Moduli(x; W^{u}(y))  $ defines a chain map
\begin{equation} \label{eq:map_AS}
  CM_{-*}(\Energy) \to CW^{*}(T^{*}_{q}Q).
\end{equation}
 For the Hamiltonian $ H= |p|^{2} $, the non-constant Hamiltonian chords with endpoints on $T^{*}_{q}Q$ are in bijective correspondence with Reeb chords with endpoints on the unit conormal, and with non-constant geodesics based at $q$.  Theorem 3.1 of \cite{ASchwarz} states that Equation \eqref{eq:map_AS} is a chain equivalence, which respects the action filtration; on the right hand side, the action of a chord $x$ is
 \begin{equation}
\Laction(x) =   \int x^{*}(\lambda) - H \circ x\, dt
 \end{equation}
where $\lambda$ is the Liouville form.    More precisely, \cite{ASchwarz} asserts that, after passing to the associated graded with respect to the action filtration, the map from Morse homology to Floer homology assigns to each geodesic the corresponding Reeb chord.

Let us now consider the situation of a plumbing:  From Equation \eqref{eq:stretched_plumbing}, we conclude that we have a codimension $0$ embedding 
\begin{equation}
  T^*Q_i^{T} \subset  \hat{M}.
\end{equation}
Let us write $\L[i]$ for the cotangent fibre at $q_i$, considered as a Lagrangian in $ \hat{M}$. Since all based geodesics for $g_{i}^{T}$ of action less than $T^2$ are contained in $ Q_i^{T} $ (see Lemma \ref{lem:short_loops_independent_of_T}), we may use the previous machinery to define a moduli space $  \Moduli(x; W^{u}(y)) $ whenever $x$ is a chord in $\hat{M}$ with endpoints on $L_i$, and $y$ a geodesic for $g_{i}^{T}$ of action less than $T$.  Note that, even in the special case when  $x$  lies in $ T^*Q_i^{T}  $, elements of this moduli space do not, a priori, have image contained in $  T^*Q_i^{T}  $.

From this moduli space, we obtain a map 
\begin{equation} \label{eq:map_to_CW}
    CM_{-*}^{T}( \Energy^{T} ) \to CW^{*}(\L[i]; r_{T}^{2}).
\end{equation}
Here $r_{T}^{2} \co \hat{M} \to [0+\infty)$ is the square of the cylindrical coordinate with respect to the decomposition $\hat{M} = M^{T} \cup \partial M^{T} \times [1,+\infty)$.  Note that this function extends the squared norm (with respect to $g_{i}^{T}$) of a cotangent vector in  $T^*Q_i^{T} $.  Let $CW^{*}_{T}(\L[i]; r_{T}^{2})  $ denote the subcomplex of the right hand side generated by chords whose action is bounded by $ T^2 $.
\begin{lem}
  The induced map on homology
\begin{equation} \label{eq:map_to_HW-bounded-action}
    HM_{0}^{T}( \Energy^{T} ) \to HW^{0}_{T}(\L[i]; r_{T}^{2}).
\end{equation}
is surjective.
\end{lem}
\begin{proof}
From the discussion above, the natural bases of  $     CM_{0}^{T}( \Energy^{T} ) $  and $CW^{0}_{T}(\L[i]; r_{T}^{2}) $ can be identified, since they both correspond to geodesics in $Q_i^{T}  $ with vanishing Jacobi index.  The exactness of $\hat{M}$ and Equation (3.1) of \cite{AS} imply that Equation (\ref{eq:map_to_CW}) preserves the action filtration: in the natural bases above, the corresponding matrix is therefore upper triangular.  Moreover, as in the case of cotangent bundles,  the diagonal entries of this matrix are $\pm 1$, corresponding to stationary solutions of the pseudo-holomorphic curve equation satisfied by elements of $  \Moduli(x; W^{u}(y)) $.

We conclude that the map
\begin{equation} \label{eq:map_to_CW-bounded-action}
    CM_{0}^{T}( \Energy^{T} ) \to CW^{0}_{T}(\L[i]; r_{T}^{2}).
\end{equation}
is an isomorphism.  Since the Morse complex is supported in non-negative degrees, whilst the Floer complex is supported in non-positive degrees, passing to cohomology in degree $0$ is obtained by taking the quotients of the two sides in Equation (\ref{eq:map_to_CW-bounded-action}) by the image of degree $1$ elements on the left, and degree $-1$ elements on the right.  We conclude that the induced map on cohomology is surjective.
\end{proof}

Since wrapped Floer homology is a symplectic invariant, the homology of the right hand side of Equation \eqref{eq:map_to_CW} is independent of $T$.  In particular, we obtain a map
\begin{equation} \label{eq:map_to_HW}
  \lim_{T} HM_{-*}^{T}( \Energy^{T} ) \to   \lim_{T} HW^{*}_{T}( \Energy^{T} ) \equiv HW^{*}(\L[i]).
\end{equation}
\begin{cor}
 Equation \eqref{eq:map_to_HW} is surjective in degree $0$. \qed
\end{cor}

From the fact that $  CM_{*}^{T}( \Energy^{T} )  $ computes the ordinary homology of the space of based loops on $Q_i^{T}   $ of bounded energy, and Equation \eqref{eq:increasing_union_loops_complement}, we compute that the homology of the limit in Equation \eqref{eq:map_to_HW} is the homology of the space of based loops in $ Q_{i}^{op} $.  Note that, since the dimension of $Q_{i}$ is greater than $1$, every loop in $Q_i$ is homotopic to a loop lying in $  Q_{i}^{op} $, and, since the dimension is greater than $2$, every homotopy between loops in $Q_i$ can be assumed to lie in $Q_i^{op}$.  We conclude
\begin{cor}
There is a surjective homomorphism
\begin{equation} \label{eq:surjection}
 \bZ[\pi_{1}(Q_i)] \to HW^{0}(\L[i]).  
\end{equation} \qed
\end{cor}
\subsubsection{Isomorphism in degree $0$}
In this section, we complete the proof of Proposition \ref{prop:compute_HW_0_cotangents}, by showing that the composition of Equations \eqref{eq:map_to_group_ring} and \eqref{eq:surjection} is the identity.

The proof is a straightforward cobordism argument, along the same lines as the case of cotangent bundles, considered in  Section 5 of \cite{wrapped-based}.  The key idea is that one can glue the domains of the half-cylinders in Equations \eqref{eq:negative_half_cylinder} and \eqref{eq:positive_half_cylinder}, to obtain a pseudo-holomorphic map whose source is an annulus of very large modular parameter.   By considering the moduli space of pseudo-holomorphic maps whose source is an annulus of arbitrary modular parameter, one obtains a homotopy between the composition of  \eqref{eq:map_to_CW} with \eqref{eq:map_CW_C_based_loops} and the natural map
\begin{equation}
   CM_{-*}^{T}( \Energy^{T} ) \to C_{-*}( \Omega_{q}(Q) )
\end{equation}
which assigns to every critical point of $\Energy^{T}   $ the chain represented by a triangulation of its descending manifold.

Passing to the direct limit over $T$, and to homology, we conclude that the composition 
\begin{equation}
  H_{-*}( \Omega_{q_i}(Q^{op}_{i}) ) = \lim_{T} HM_{-*}^{T}( \Energy^{T} ) \to HW^{*}(\L[i]) \to H_{-*}( \Omega_{q_i}(Q_i) )
\end{equation}
is homotopic to the map 
\begin{equation}
  H_{-*}( \Omega_{q_i}(Q^{op}_{i}) ) \to H_{-*}( \Omega_{q_i}(Q_i) )
\end{equation}
induced by the inclusion $Q^{op}_{i} \subset   \Omega_{q_i}(Q_i)  $.  By our assumption on the dimension of $Q_i$, this map is an isomorphism in degree $0$.  We conclude:
\begin{lem}
Equations \eqref{eq:map_to_group_ring} and \eqref{eq:surjection} are isomorphisms.
\end{lem}
\begin{proof}
Since the compositions of Equations \eqref{eq:map_to_group_ring} and \eqref{eq:surjection} is the identity, the first map must be surjective, and the second injective.  However, we already know that Equation \eqref{eq:surjection} is a surjection, so it is therefore an isomorphism.  Since the inverse to an isomorphism is necessarily injective, we conclude that Equations \eqref{eq:map_to_group_ring} is an isomorphism as well.
\end{proof}


\end{document}